\documentclass[12pt,a4paper,openbib]{article}
\usepackage[utf8]{inputenc}
\usepackage[T1]{fontenc}
\usepackage[french]{babel}
\usepackage{amsmath}
\usepackage{amsfonts}
\usepackage{dsfont}
\usepackage{tikz}
\usetikzlibrary{cd}
\usepackage{amssymb}
\usepackage{amsthm}
\usepackage{lmodern}
\usepackage{fancyhdr}
\usepackage{lettrine}
\usepackage{bm}
\usepackage{enumerate}
\usepackage{multicol}
\usepackage[hidelinks]{hyperref}
\usepackage[nottoc, notlof, notlot]{tocbibind}
\usepackage[left=2cm,right=2cm,top=2cm,bottom=2cm]{geometry} 
\usepackage{graphicx}
\usepackage{mathtools}
\usepackage{wasysym}
\usepackage{stmaryrd}
\usepackage[all,cmtip]{xy}
\usepackage{enumitem}
\usepackage{relsize} 
\usepackage{mathrsfs}

\makeatletter

\newcommand{\address}[1]{\gdef\@address{#1}}
\newcommand{\email}[1]{\gdef\@email{\url{#1}}}
\newcommand{\@endstuff}{\par\vspace{\baselineskip}\noindent\small
\begin{tabular}{@{}l}\scshape\@address\\\textit{E-mail address:} \@email\end{tabular}}
\AtEndDocument{\@endstuff}
\makeatother

\author{Damien Junger\footnote{This work has been written in a great part during the author PhD thesis at the ENS Lyon. His work are currently funded by the Deutsche Forschungsgemeinschaft (DFG, German Research Foundation) under Germany's Excellence Strategy EXC 2044–390685587, Mathematics Münster: Dynamics–Geometry–Structure.}}
\address{Mathematisches Institut, Universität Münster,\\ Fachbereich Mathematik und Informatik der Universität Münster,  Orléans-Ring 10, 48149 Münster, Germany.}
\email{djunger@uni-muenster.de}
\title{Cohomologie de de Rham du revêtement modéré de la tour de Lubin-Tate}

\newtheorem{theointro}{Th\'eor\`eme}

\newtheorem{theo}{Th\'eor\`eme}[section]
\newtheorem{lem}[theo]{Lemme}
\newtheorem{lemintro}[theointro]{Lemme}

\newtheorem{prop}[theo]{Proposition}
\newtheorem{propintro}[theointro]{Proposition}

\theoremstyle{definition}
\newtheorem{defi}[theo]{D\'efinition}
\newtheorem{claim}[theo]{Fait}

\theoremstyle{remark}
\newtheorem{rem}[theo]{Remarque}

\DeclareMathOperator{\cind}{c-ind}
\DeclareMathOperator{\ind}{ind}
\DeclareMathOperator{\card}{Card}
\DeclareMathOperator{\spec}{Spec}
\DeclareMathOperator{\spf}{Spf}
\DeclareMathOperator{\spg}{Sp}

\DeclareMathOperator{\mat}{M}
\DeclareMathOperator{\gln}{GL}
\DeclareMathOperator{\pgln}{PGL}

\DeclareMathOperator{\dl}{DL}

\DeclareMathOperator{\stab}{Stab}
\DeclareMathOperator{\fix}{Fix}

\DeclareMathOperator{\hhh}{H}
\DeclareMathOperator{\wwwww}{W}

\DeclareMathOperator{\homm}{Hom}

\DeclareMathOperator{\aut}{Aut}

\DeclareMathOperator{\gal}{Gal}
\DeclareMathOperator{\id}{Id}
\DeclareMathOperator{\tr}{Tr}

\DeclareMathOperator{\nr}{Nr}
\DeclareMathOperator{\lt}{LT}

\DeclareMathOperator{\coker}{coker}

\DeclareMathOperator{\jl}{JL}
\DeclareMathOperator{\gr}{Gr}
\DeclareMathOperator{\art}{Art}
\DeclareMathOperator{\bl}{Bl}
\DeclareMathOperator{\proj}{Proj}
\DeclareMathOperator{\sym}{Sym}
\DeclareMathOperator{\codim}{codim}
\DeclareMathOperator{\Irr}{Irr}

 \newcommand{\iso}{\stackrel{\sim}{\fl}}

\font\tengoth=eufb10
\font\sevengoth=eufb7
\font\fivegoth=eufb5

\newfam\gothfam
\textfont\gothfam=\tengoth
\scriptfont\gothfam=\sevengoth
\scriptscriptfont\gothfam=\fivegoth

\def\A{{\mathbb{A}}}
\def\B{{\mathbb{B}}}

\def\F{{\mathbb{F}}}

\def\H{{\mathbb{H}}}

\def\M{{\mathbb{M}}}
\def\N{{\mathbb{N}}}

\def\P{{\mathbb{P}}}
\def\Q{{\mathbb{Q}}}

\def\Z{{\mathbb{Z}}}

\def\BC{{\mathcal{B}}}
\def\CC{{\mathcal{C}}}

\def\FC{{\mathcal{F}}}
\def\HC{{\mathcal{H}}}

\def\GC{{\mathcal{G}}}

\def\MC{{\mathcal{M}}}
\def\OC{{\mathcal{O}}}

\def\TC{{\mathcal{T}}}

\def\XC{{\mathcal{X}}}
\def\YC{{\mathcal{Y}}}

\def\SG{{\mathfrak{S}}}

\def\XG{{\mathfrak{X}}}

\def\mG{{\mathfrak{m}}}

\def\If{{\mathscr{I}}}

\def\Jf{{\mathscr{J}}}

\def\Nf{{\mathscr{N}}}

\def\Of{{\mathscr{O}}}
\def\Pf{{\mathscr{P}}}

\def\bar#1{\overline{#1}}

\def\hetc#1{{\rm H}^{#1}_{{\rm ét},c}}

\def\hdr#1{{\rm H}^{#1}_{\rm dR}}
\def\hdrc#1{{\rm H}^{#1}_{{\rm dR},c}}
\def\hrig#1{{\rm H}^{#1}_{\rm rig}}
\def\hrigc#1{{\rm H}^{#1}_{{\rm rig},c}}

\def\et{\text{ et }}
\def\si{\text{Si }}
\def\sinon{\text{Sinon}}
\def\and{\text{ and }}

\def\avec{\text{ avec }}

\def\fl{\rightarrow}

\def\fln#1#2{\xrightarrow[#2]{#1}}

\def\inter#1#2#3{\bigcap\limits_{{\substack{#2}}}^{#3}{#1}}
\def\uni#1#2#3{\bigcup\limits_{{\substack{#2}}}^{#3}{#1}}

\begin{document}
\maketitle

\begin{abstract}
In this article, we study the De Rham cohomology of the first cover in the Lubin-Tate tower. In particular, we get a purely local proof that the supercuspidal part realizes the local Jacquet-Langlands correspondence for ${\rm GL}_n$ by comparing it to the rigid cohomology of some Deligne-Lusztig varieties. The representations obtained are analogous to the ones appearing in the $\ell$-adic cohomology if we forget the action of the Weil group. The proof relies on the generalization of an excision result of Grosse-Kl\"onne and on the existence of a semi-stable model constructed by Yoshida for which we give a more explicit description.
\end{abstract}

\tableofcontents
 
\section{Introduction}

Cet article fait suite à \cite{J4} où l'on détermine la partie supercuspidale de la cohomologie de Rham pour le premier revêtement de la tour de Drinfeld. Dans les articles fondateurs \cite{lt1,lt2,dr1,dr2}, deux tours de revêtements rigides-analytiques  $(\MC_{LT}^n)_n$ ont été construits dont la cohomologie étale $l$-adique a un lien profond avec les correspondances de Jacquet-Langlands et de Langlands locales. La première, la tour de Drinfeld, est une famille dénombrable de revêtements successifs de l'espace symétrique de Drinfeld à savoir $\P_K^d \backslash \bigcup_{H \in \HC} H$ le complémentaire dans l'espace projectif sur $K$ une extension finie de $\Q_p$ des hyperplans $K$-rationnels. La seconde, la tour de Lubin-Tate, constitue une famille de revêtements sur la boule unité rigide $\mathring{\B}_K$. L'étude des propriétés cohomologiques de ces espaces a culminé dans les travaux \cite{dr1, cara3, harrtay, falt, fargu, dat1, dat0, mied, Scholze1}...  où il a été montré que la partie supercuspidale de la cohomologie $l$-adique réalisait à la fois les correspondances de  Jacquet-Langlands et de Langlands locale. Le but de ce travail ainsi que de \cite{J4} est d'obtenir des résultats similaires pour les cohomologies $p$-adiques avec l'espoir d'exhiber des versions $p$-adiques des correspondances  de Langlands locale  encore conjecturales. En guise de première étape à l'établissement de ce programme conséquent, nous nous concentrons uniquement dans ces travaux sur la cohomologie de de Rham où nous nous attendons à obtenir les mêmes représentations que pour la cohomologie $l$-adique où l'on a oublié l'action du groupe de Weil. En particulier, nous pouvons seulement exhiber la correspondance de  Jacquet-Langlands. 

Même si certaines méthodes en $l$-adique ne peuvent être adaptées à l'étude que nous voulons mener et si la cohomologie de de Rham de l'ensemble des deux tours semble inaccessible pour le moment, quelques cas particuliers ont été établis comme le cas de dimension $1$ du coté Drinfeld (quand $K=\Q_p$) \cite{brasdospi}, ($K$ quelconque)\cite{GPW1}, du coté Lubin-Tate \cite{GPW1} et récemment en dimension quelconque pour le premier revêtement du côté Drinfeld \cite{J4} (voir aussi \cite{scst,iovsp,ds} pour le niveau $0$ du côté Drinfeld\footnote{En niveau $0$ du coté Lubin-Tate, le résultat est immédiat car la boule unitée ouverte  rigide n'a pas de cohomologie.}). Dans cet article, nous prouvons le résultat pour le premier revêtement du coté Lubin-Tate. L'espoir de prolonger les résultats du côté Drinfeld aux espaces analogues du coté Lubin-Tate est motivé par le lien profond qui existe entre la géométrie des deux tours. Ce lien s'illustre dans les travaux \cite{falt2, fafa} où il a été montré que chaque tour en niveau infini à une structure d'espace perfecto\"ide et que ces deux espaces sont isomorphes.

Avant d'énoncer le résultat principal de cet article, introduisons les représentations qui apparaissent dans l'énoncé. Soit $K$ une extension finie de $\Q_p$ d'anneau des entiers $\OC_K$, d'uniformisante $\varpi$  et de corps résiduel $\F=\F_q$. Notons $C=\widehat{\bar{K}}$ le complété d'une clôture alg\'ebrique de $K$ et $\breve{K}= \widehat{K^{nr}}$ le compl\'eté de l'extension maximale non ramifi\'ee dans $\bar{K}$.  On note $(\MC_{LT}^n)_n$ les revêtements   de dimension  $d$ sur $\breve{K}$ de la tour de Lubin-Tate. Chacun de ces espaces rigide-analytiques admettent des actions des groupes $\gln_{d+1}(\OC_K)=G^\circ$, $D^*$ avec $D$ l'algèbre à division sur   $K$ de dimension $(d+1)^2$ et d'invariant $1/(d+1)$, $W_K$ le groupe de Weil de $K$ qui commute entre elles. De plus, ces revêtements se décompose sous la forme $\MC_{LT}^n=\lt^n\times\Z$ et on peut aisément relier les cohomologie de $\MC_{LT}^n$ et de $\lt^n$.

  Étant donné un caractère primitif $\theta$ du groupe $\F_{q^{d+1}}^*$, on peut lui associer les représentations suivantes :

\begin{enumerate}
\item Sur le groupe $\gln_{d+1}(\F_q)$, $\dl(\theta)$ sera la repr\'esentation cuspidale associ\'ee \`a $\theta$ via la correspondance de Deligne-Luzstig. On pourra la voir comme une repr\'esentation de $G^\circ $ par inflation. 
\item $\tilde{\theta}$ sera le caract\`ere de $I_K (\varphi^{d+1})^{\Z}\subset W_K$ par le biais de $I_K \to I_K/ I_{K_N}\cong \F_{q^{d+1}}^*\xrightarrow{\theta} \widehat{\bar{K}}^*$ avec  $K_N=K(\varpi^{1/N})$. On impose de plus $\tilde{\theta}(\varphi^{d+1})=(-1)^d q^{\frac{d(d+1)}{2}}$.  
\end{enumerate}
Nous donnons maintenant les repr\'esentations associ\'ees sur les groupes $D^*$, $W_K$. Posons :
\begin{enumerate}
\item $\rho(\theta):= \cind_{\OC_D^*\varpi^\Z}^{D^*} \theta$ o\`u $\theta$ est vue comme une $\OC_D^* \varpi^{\Z}$-repr\'esentation via $\OC_D^* \varpi^{\Z} \to \OC_D^* \to \F_{q^{d+1}}^*$. 
\item $\wwwww(\theta):= \cind_{I_K (\varphi^{d+1})^{\Z}}^{W_K} \tilde{\theta}$. 
\end{enumerate}

Le théorème principal de l'article est le suivant :
\begin{theointro}
\label{theointroprinc}
  Pour tout caractère primitif $\theta: \F_{q^{d+1}}^*\to C^*$, il existe des isomorphismes de $G^\circ$-représentations 
 \[\homm_{D^*}(\rho(\theta), \hdrc{i}((\MC_{LT, C}^1/ \varpi^{\Z}))){\cong} \begin{cases} \dl(\theta)^{d+1} &\text{ si } i=d \\ 0 &\text{ sinon} \end{cases}.\]

\end{theointro}
Rappelons que les représentations $\rho(\theta)$ considéré parcourt l'ensemble des représentations irréductibles de niveau $0$ de $D^*$ de caract\`ere central trivial sur $\varpi^{\Z}$ dont l'image par la correspondance de Jacquet-Langlands est   supercuspidale. 
 
Expliquons dans les grandes lignes la stratégie de la preuve. L'un des grands avantages à travailler sur les premiers revêtements de chacune des tours est que la géométrie des deux espaces considérés reste encore accessible. Cela nous permet alors d'espérer pouvoir appliquer les arguments purement locaux  des thèses de \cite{yosh} (côté Lubin-Tate) et \cite{wa} (côté Drinfeld). En niveau supérieur, la géométrie se complique grandement et nous doutons qu'une stratégie similaire puisse fonctionner. Du côté Drinfeld, cela s'illustre par l'existence d'une équation globale au premier revêtement (voir \cite{J3}). Du côté Lubin-Tate, la situation est encore meilleure car on a l'existence d'un modèle semi-stable généralisé (voir plus bas) construit par Yoshida. Cette propriété jouera un rôle clé dans la stratégie que nous avons entrepris et nous expliquons maintenant son utilité.

Soit $\XC$ un schéma formel $p$-adique, il est de réduction semi-stable généralisée si, Zariski-localement sur $\XC$, on peut trouver un morphisme étale vers ${\rm Spf}(\OC_K\langle X_1,...,X_n\rangle/
   (X_1^{\alpha_1}...X_r^{\alpha_r}-\varpi))$ pour certains $r\leq n$ et $\alpha_i\geq 1$ (ou $\varpi$ est une uniformisante de $K$).  Observons que pour tout schéma formel $\XC$ vérifiant cette propriété, l'anneau local complété en un point fermé est de la forme 
\[
\OC_{\breve{K}} \llbracket T_0,\ldots, T_d  \rrbracket /(T_0^{e_0}\cdots T_r^{e_r}- \varpi )
\]
avec $r\leq d$ et $e_i$ premier à $p$. 

Nous dirons alors que $\XG$ est ponctuellement semi-stable généralisé si pour tous point fermés, l'anneau local complété est de cette forme. Il s'agit d'une notion un peu plus faible car on peut s'autoriser des espaces qui ne sont pas $p$-adiques\footnote{Par exemple, $\spf(\OC_{\breve{K}} \llbracket T_0,\ldots, T_d  \rrbracket /(T_0^{e_0}\cdots T_r^{e_r}- \varpi ))$ est ponctuellement semi-stable généralisé mais pas semi-stable généralisé.}. 
Le résultat suivant qui a en premier été prouvé par Grosse-Kl\"onne  dans \cite[Theorem 2.4.]{GK1} pour le cas semi-stable puis dans \cite[Théorème 5.1.]{J4} pour le cas semi-stable généralisé, est le point de départ de notre argument.
\begin{propintro}

Étant donné un schéma formel  semi-stable généralisé $\XC$  avec pour décomposition en composantes irréductibles $\XC_s=\bigcup_{i\in I} Y_i$, pour toute partie finie $J\subset I$, la flèche naturelle de restriction 
\[\hdr{*} (\pi^{-1}(]Y_J[_\XC))\fln{\sim}{} \hdr{*} (\pi^{-1}(]Y_{J}^{lisse}[_\XC))\]
est un isomorphisme pour $Y_J=\inter{Y_j}{j\in J}{}$ et $Y_J^{lisse}=Y_{J}\backslash \bigcup_{i\notin J}Y_i$.

\end{propintro}

Ce résultat nous permet de mettre en place l'heuristique suivante. Étant donné un schéma formel semi-stable généralisé, on a recouvrement de la fibre générique par les tubes des composantes irréductibles $(]Y_i[)_{i\in I}$  de la fibre spéciale. On veut alors appliquer la suite spectrale de Cech à ce recouvrement. On se ramène alors à calculer les cohomologies des intersections $]Y_J[$ pour $J\subset I$ et donc de $]Y_J^{lisse}[$. La géométrie de ces derniers plus simples et les exemples apparaissant dans cet article  admettent par exemple des modèles entiers lisses ce qui permet de calculer leurs cohomologies par théorème de comparaison avec la cohomologie rigide de leur fibre spéciale. 

Pour pouvoir mettre en place cette stratégie dans ce cas, nous devons comprendre le modèle semi-stable construit par Yoshida  et décrire explicitement les différentes composantes irréductibles de la fibre spéciale. En fait, le premier revêtement admet un modèle naturel $Z_0$ provenant de son interprétation modulaire. Ce dernier n'est pas encore semi-stable et nous y résolvons les singularités en éclatant successivement des fermés bien choisis. Plus précisément, il construit une famille de fermés dans la fibre spéciale $(Y_M)_M$ indexés par les sous-espaces vectoriels propres $M\subsetneq \F^{d+1}$ qui vérifie $Y_M\cap Y_N=Y_{M\cap N}$. Lorsque $M$ parcourt l'ensemble des hyperplans, les fermés $Y_M$ décrivent l'ensemble des composantes irréductibles de $Z_0$ et on peut exhiber une stratification de la fibre spéciale par les fermés
\[
Y^{[h]}:=\bigcup_{N :\dim N=h} Y_N
\]
On construit alors une suite de modèles $Z_0, \cdots, Z_d$ en éclatant successivement le long des fermés de cette stratification. De plus, à chaque modèle, on peut construire de manière similaire une famille de fermés $(Y_{M,h})_h$ de $Z_h$  avec $Y_{M,h}=Y_M$ en prenant des transformées propres ou strictes suivant les cas (voir la définition \ref{defitrans} pour plus de précisions). Cette famille est essentielle pour comprendre la géométrie de la fibre spéciale et le résultat suivant illustre ce principe.
\begin{theointro}
On a les points suivants :
\begin{enumerate}
\item $Z_d$ est ponctuellement semi-stable généralisé.
\item Les composantes irréductibles de la fibre spéciale de $Z_h$ sont les fermés  de dimension $d$ suivant   $(Y_{M,h})_{M: \dim M\in \left\llbracket 0,h-1\right\rrbracket\cup\{d\}}$. 
\item Les intersections non-vides  de composantes irréductibles de $Z_h$ sont en bijection avec les drapeaux $M_1\subsetneq \cdots \subsetneq M_k$  tels que $\dim M_{k-1}< h$ via l'application \[M_1\subsetneq \cdots \subsetneq M_k \mapsto \bigcap_{1\leq i \leq k} Y_{M_i,h}.\]
\item  Si $Y_{M,h}$ est une composante irréductible de $Z_{h,s}$ avec $\dim M\neq d$, alors les morphismes naturels $\tilde{p}_i$ avec $i=h+1,\ldots,  d$ induisent  des isomorphismes $  Y_{M,d}^{lisse}  \cong\cdots \cong Y_{M,h+1}^{lisse} \cong Y_{M,h}^{lisse}$. 
\item Le changement de base du tube $]Y_{\{0\},d}^{lisse}[_{Z_d} \otimes \breve{K}(\varpi_N)\subset {\rm LT}_1 \otimes \breve{K}(\varpi_N)$ admet un modèle lisse dont la fibre spéciale est isomorphe à la variété de Deligne-Lusztig $\dl_{\bar{\F}}$ associée à $\gln_{d+1}$ et à l'élément de Coxeter $(1,\cdots, d)\in \SG_d$ (ici, $N=q^{d+1}-1$ et $\varpi_N$ est le choix d'une racine $N$-ième de $\varpi$).
\end{enumerate}

\end{theointro}

Grâce à ce résultat et au théorème d'excision cité précédemment, nous pouvons alors prouver :
\begin{propintro}

On a un isomorphisme $\gln_{d+1}(\OC_K)$-équivariant :
\[\hdr{*} (\lt^1/\breve{K}_N)\cong \hrig{*} (\dl^d_{\bar{\F}}/\breve{K}_N)\] Par dualité de Poincaré, on a un isomorphisme semblable pour les cohomologies à support compact.

\end{propintro}

\begin{rem}

Notons que le modèle $Z_d$ obtenu est semi-stable généralisé au sens faible et nous ne pouvons appliquer directement le théorème d'excision. Cette difficulté est surmontée par une astuce déjà présente dans les travaux de Yoshida  où $Z_d$ est plongé dans un modèle entier $\hat{{\rm Sh}}$ bien choisi d'une variété de Shimura qui est cette fois-ci semi-stable généralisée au sens fort (et donc $p$-adique). On peut alors un voisinage étale $U\subset \hat{{\rm Sh}}$ de $Z_d$ qui est $p$-adique et semi-stable généralisé et sur lequel on peut appliquer le théorème d'excision.

\end{rem}

L'intérêt principal du résultat précédent réside dans le fait que les variétés de Deligne-Lusztig $\dl^d_{\F}$ réalisent les correspondances de Green desquels peuvent se déduire  les correspondances de Jacquet-Langlands et de Langlands locale pour les représentation supercuspidales de niveau $0$. La dernière difficulté consiste à vérifier que les actions des groupes $G^\circ:=\gln_{d+1}(\OC_K)$, $D^*$ et $W_K$ sur $\lt^1$ induisent les bons automorphismes sur la variété de Deligne-Lusztig  pour pouvoir appliquer cette correspondance. Plus précisément, l'espace algébrique $\dl^d_{\F}$ est un revêtement galoisien de $\Omega^d_{\F}:=\P_{\F}^d \backslash \bigcup_{H\in \P^d(\F)} H$ tel que $\gal (\dl^d_{\F}/\Omega^d_{\F})=\F_{q^{d+1}}$ est cyclique d'ordre premier à $p$. De plus, il admet une action de $G^\circ$ qui commute à la projection du revêtement.  Il s'agit alors de prouver que les actions des groupes  $\OC_D^*$ et $I_K$ sur $\lt^1$ se transporte à $\dl^d_{\F}$ et induisent des isomorphismes naturels :
\[
\OC_D^* /(1+\Pi_D \OC_D)\cong \gal (\dl^d_{\F}/\Omega^d_{\F})\cong I_K/I_{K_N}
\]
avec $N=q^{d+1}-1$ et $K_N=K(\varpi^{1/N})$.

Pour le groupe $I_K$, cela a été déjà fait par Yoshida dans sa thèse et le cas de $\OC_D^*$ est essentiel pour étudier la correspondance de Jacquet-Langlands. Expliquons comment nous parvenons à traiter ce cas. Comme les actions de $I_K$ et de $\OC_D^*$ commutent entre elles, on obtient une action de $\OC_D^*$ sur la base $\Omega^d_{\F}$ qui commute avec celle de $G^\circ$. Nous nous ramenons alors à prouver le résultat technique suivant :
\begin{lemintro}\label{lemintroautdl}
On a 
\[ \aut_{\gln_{d+1}(\F)}(\Omega^d_{\bar{\F}})= \{ 1 \}. \]
\end{lemintro}

En particulier, cette annulation montre que l'on a bien une flèche naturelle $\OC_D^* \to \gal (\dl^d_{\F}/\Omega^d_{\F})$ qui induit une flèche $\OC_D^* /(1+\Pi_D \OC_D)\to \gal (\dl^d_{\F}/\Omega^d_{\F})$ car le groupe $1+\Pi_D \OC_D$ est $N$-divisible alors que $ \gal (\dl^d_{\F}/\Omega^d_{\F})$ est de $N$-torsion. Par finitude et égalité des cardinaux, il suffit de montrer que cette flèche est injective.

Mais par description explicite de la cohomologie $l$-adique de $\lt^1$ et de $\dl^d_{\F}$  en tant que $\bar{\Q}_l[\F_{q^{d+1}}^*]$-module, l'identité est le seul élément qui agit par un automorphisme de trace non-nulle. Par formule de Lefschetz, il suffit alors de prouver que les éléments réguliers elliptiques de $\OC_D^*$ qui ne sont pas dans $1+\Pi_D \OC_D$ n'ont aucun point fixe dans $\lt^1(C)$. Cela se prouve en étudiant l'application des périodes $\pi_{GH}: \lt^1(C)\to \P^d(C)$ de Gross-Hopkins (cf \cite{grho}). Les points fixes dans le but sont faciles à déterminer et leur fibre est explicite. On peut alors montrer qu'il n'y a aucun point fixe dans chacune de ces fibres par calcul direct.

En particulier, le raisonnement précédent permet aussi de réaliser  les correspondances de Jacquet-Langlands et de Langlands locale  pour la cohomologie $l$-adique donnant ainsi une version plus forte du résultat principal de la thèse de Yoshida \cite{yosh} :     

\begin{theointro}
\label{theointroprincet}
Fixons un isomorphisme $C\cong \bar{\Q}_l$.  Pour $l$ premier à $p$, tout caractère primitif $\theta: \F_{q^{d+1}}^*\to C^*$, il existe des isomorphismes de $G^\circ\times W_K$-représentations 
 \[\homm_{D^*}(\rho(\theta), \hetc{i}((\MC_{LT}^1/ \varpi^{\Z}),\bar{\Q}_l)\otimes C)\underset{G^{\circ}\times W_K}{\cong}  \begin{cases} \dl(\theta) \otimes \wwwww(\theta) &\text{ si } i=d \\ 0 &\text{ sinon} \end{cases}\]

\end{theointro}

\subsection*{Remerciements}

Le présent travail a été, avec \cite{J1,J2,J3,J4}, en grande partie réalisé durant ma thèse à l'ENS de lyon,  et a pu bénéficier des nombreux conseils et de l'accompagnement constant de mes deux maîtres de thèse Vincent Pilloni et Gabriel Dospinescu. Je les en remercie très chaleureusement. Je tenais aussi à exprimer ma reconnaissance envers  Juan Esteban Rodriguez Camargo pour les nombreuses discussions sur les éclatements qui ont rendu possible la rédaction de ce manuscrit.

\subsection*{Conventions\label{paragraphconv}}
Dans tout l'article, on fixe $p$ un premier. Soit $K$ une extension finie de $\Q_p$ fixée, $\mathcal{O}_K$ son anneau des entiers, $\varpi$ une uniformisante et $\F=\F_q$ son corps r\'esiduel. On note $C=\hat{\bar{K}}$ une complétion d'une clôture algébrique de $K$ et $\breve{K}$ une complétion de l'extension maximale non ramifiée de $K$. Soit $L\subset C$ une extension complète de $K$ susceptible de varier, d'anneau des entiers $\mathcal{O}_L$, d'idéal maximal  $\mG_L$ et de corps r\'esiduel $\kappa$. $L$ pourra être par exemple spécialisé en $K$, $\breve{K}$ ou $C$.

Soit $S$ un $L$-espace analytique, on note $\A^n_{ S}=\A^n_{ L}\times S$ l'espace affine  sur $S$ 
  L'espace $\B^n_S$ sera la boule unité et les boules ouvertes seront notées  $\mathring{\B}^n_S$. 




Si $X$ est un affinoïde sur $L$,  on notera $X^\dagger$ l’espace surconvergent associé. Dans ce cas,  le complexe de de Rham $\Omega_{X/L}^\bullet$ (resp. le complexe de de Rham surconvergent $\Omega_{X^\dagger/L}^\bullet$) qui calcule la cohomologie de de Rham $\hdr{*}(X)$ (resp. la cohomologie de de Rham surconvergent $\hdr{*}(X^\dagger)$). Quand $X$ est quelconque, ces cohomologies seront formées à partir de l’hypercohomologie de ces complexes. Par \cite[Proposition 2.5]{GK1}, le théorème $B$ de Kiehl \cite[Satz 2.4.2]{kie} et la suite spectrale de Hodge-de Rham, si $X$ est Stein, ces cohomologies sont encore calculées à  partir de leur complexe respectif\footnote{En cohomologie de de Rham (non surconvergente), l'hypothèse $X$ quasi-Stein  suffit.}. Les deux cohomologies coïncident si $X$ est partiellement propre (par exemple Stein).

La cohomologie rigide d'un schéma algébrique $Y$ sur $\kappa$ sera notée $\hrig{*} (Y/L)$. Si $X$ est un espace rigide sur $L$ et $Y$ un schéma algébrique sur $\kappa$, $\hdrc{*} (X^\dagger)$ et $\hrigc{*} (Y/L)$ seront les cohomologies de $X^\dagger$ et de $Y$ à  support compact. On rappelle la dualité de Poincaré :
\begin{theo} \label{theodualitepoinc}
\begin{enumerate}
\item \cite[proposition 4.9]{GK1} Si $X$ est un $L$-affinoïde lisse  de dimension pure $d$,  on a \[\hdr{i} (X^\dagger)\cong \hdrc{2d-i} (X^\dagger)^\lor \et \hdrc{i} (X^\dagger)\cong \hdr{2d-i} (X^\dagger)^\lor\]

\item \cite[proposition 4.11]{GK1}  Si $X$ est un $L$-espace lisse et Stein de dimension pure $d$,  on a \[\hdr{i} (X^\dagger)\cong \hdrc{2d-i} (X^\dagger)^\lor \et \hdrc{i} (X^\dagger)\cong \hdr{2d-i} (X^\dagger)^\lor\]
\item \cite[théorème 2.4]{bert1} Si $Y$ est un schéma algébrique lisse sur $\kappa$ de dimension $d$, on a pour tout $i$ \[\hrig{i} (Y/L)\cong \hrigc{2d-i} (Y/L)^\lor \et \hrigc{i} (Y/L)\cong \hrig{2d-i} (Y/L)^\lor\]
\end{enumerate}

\end{theo}
 Nous donnons un théorème de comparaison :
\begin{theo}\label{theopurete}
 \cite[proposition 3.6]{GK3}  
Soit $\XC$ un schéma affine formel sur $\spf (\OC_L)$ de fibre spéciale $\XC_s$ et de fibre générique $\XC_\eta$. Supposons $\XC$ lisse, alors on a un isomorphisme fonctoriel \[\hdr{*}(\XC_\eta^\dagger)\cong \hrig{*}(\XC_s)\]

\end{theo}

Dans ce paragraphe, $L$ est non-ramifié sur $K$. Soit $\XC$ un schéma formel topologiquement de type fini sur $\spf  (\OC_L)$ de fibre g\'en\'erique $\XC_\eta$ et de fibre sp\'eciale $\XC_s$. On a une flèche de spécialisation $\spg: \XC_\eta \rightarrow \XC_s$ et pour tout fermé $Z\subset \XC_s$, on note $]Z[_{\XC}$ l'espace analytique $\spg^{-1} (Z)\subset \XC_\eta$.  

\begin{defi}

On dit que $\XC$ est faiblement de réduction  semi-stable généralisé s'il admet un recouvrement $\XC=\uni{U_t}{t\in T}{}$ et un jeu de morphisme fini étale  \[\varphi_t : U_t\to  \spf (\OC_L\left\langle x_1,\cdots , x_d\right\rangle/(x_1^{\alpha_1}\cdots x_r^{\alpha_r}-\varpi )).\]

\end{defi}

  Quitte à rétrécir les ouverts $U_t$ et à prendre $r$ minimal, on peut supposer que les composantes irréductibles  $\bar{U}_t$ sont les $V(\bar{x}^*_i)$ pour $i\leq r$ avec $\bar{x}^*_i =\bar{\varphi}_t (\bar{x}_i)$ et   ont multiplicités $\alpha_i$. 

Observons que pour schéma formel semi-stable $\XC$, on vérifie aisément que $\XC$ est ponctuellement de réduction semi-stable généralisé au sens suivant. 
\begin{defi}

Un schéma formel est ponctuellement de réduction semi-stable généralisé si l'anneau local complété en chacun des points fermés et de la forme
\[\OC_L\left\llbracket x_1,\cdots , x_d\right\rrbracket /(x_1^{\alpha_1}\cdots x_r^{\alpha_r}-\varpi ).\]

\end{defi}

Cette notion plus faible  nous permet par exemple de considérer des schémas formels qui ne sont pas localement des complétions $p$-adiques de schémas algébriques de type fini sur $\OC_L$. Par exemple, l'espace lui-même $\spf(\OC_L\left\llbracket x_1,\cdots , x_d\right\rrbracket /(x_1^{\alpha_1}\cdots x_r^{\alpha_r}-\varpi ))$ est ponctuellement  semi-stable généralisé alors qu'il n'est pas  semi-stable généralisé d'après la remarque précédent.  Le résultat  suivant suggère que sous l'hypothèse d'être localement la complétion $p$-adique d'un schéma algébrique de type fini sur $\OC_L$, les deux notions de semi-stabilité  sont  en quelque sorte équivalentes. 

\begin{prop}

Si $\XC$ est un schéma formel ponctuellement  semi-stable généralisé qui  admet une immersion ouverte vers un schéma formel $\YC$ qui est  localement la  complétion $p$-adique de schémas algébriques de type fini sur $\OC_L$, alors il existe un voisinage étale $U$ de $\XC$ dans l'espace ambiant $\YC$ qui est  semi-stable généralisé.

\end{prop}

\begin{proof}
Voir la preuve de \cite[Propositions 4.8. (i)]{yosh}
\end{proof}

L'intérêt des espaces semi-stables que nous avons introduits provient du fait que leur géométrie fait naturellement apparaître des recouvrements dont on peut espérer calculer la cohomologie des intersections grâce à \ref{theopurete} et au résultat qui va suivre. Pour pouvoir énoncer ce dernier, nous introduisons quelques notations pour un schéma formel semi-stable généralisé $\XG$ (au sens le plus fort) dont la fibre spéciale admet la décomposition en composantes irréductibles $\XC_s=\bigcup_{i\in I} Y_i$. Pour toute partie finie $J\subset I$ de l'ensemble des composantes irréductibles, on note $Y_J=\inter{Y_j}{j\in J}{}$ et $Y_J^{lisse}=Y_{J}\backslash \bigcup_{i\notin J}Y_i$. Le résultat est le suivant 

\begin{theo}\label{theoexcision}

Étant donné un schéma formel  semi-stable généralisé $\XC$ comme précédemment avec pour décomposition en composantes irréductibles $\XC_s=\bigcup_{i\in I} Y_i$, pour toute partie finie $J\subset I$, la flèche naturelle de restriction 
\[\hdr{*} (\pi^{-1}(]Y_J[_\XC))\fln{\sim}{} \hdr{*} (\pi^{-1}(]Y_{J}^{lisse}[_\XC))\]
est un isomorphisme.

\end{theo}

\begin{proof}
Il s'agit de la généralisation du résultat pour le cas semi-stable \cite[Theorem 2.4.]{GK2} au cas semi-stable généralisé réalisée dans \cite[Théorème 5.1.]{J4}.
\end{proof}

\section{La tour de revêtements}


Nous allons d\'efinir la tour de Lubin-Tate construite dans \cite[section 1 et 4]{dr1}. 
Soit $\CC$ la sous-cat\'egorie pleine des $\OC_{\breve{K}}$-alg\`ebres locales noeth\'eriennes et compl\`etes $A$ telles que le morphisme naturel $\OC_{\breve{K}}/ \varpi\OC_{\breve{K}} \to A/ \mG_A$ soit un isomorphisme. On consid\`ere, pour $A$ un objet de $\CC$, l'ensemble des $\OC_K$-modules formels $F$ sur $A$ modulo isomorphisme. On note $X+_{F} Y\in A\left\llbracket X,Y\right\rrbracket$ la somme et $[ \lambda ]_F X\in A\left\llbracket X\right\rrbracket$ la multiplication par $\lambda$ dans $F$. 

Si $A$ est de caract\'eristique $p$, on appelle hauteur le plus grand entier $n$ (possiblement infini) tel que $[ \varpi ]_F$ se factorise par ${\rm Frob}_q^n$. On a le r\'esultat classique (voir par exemple \cite[chapitre III, §2, théorème 2]{frohlforgr})  : %

\begin{prop}
Si $A= \bar{\F}$, la hauteur est un invariant total i.e. deux $\OC_K$-modules formels sont isomorphes si et seulement si ils ont la m\^eme hauteur.  
\end{prop}

Fixons un repr\'esentant "normal"  $\Phi_d$ de hauteur $(d+1)$ tel que : 
\begin{enumerate}
\item $[ \varpi ]_{\Phi_d} X\equiv X^{q^{d+1}} \pmod{X^{q^{d+1}+1}}$, 
\item $X +_{\Phi_d} Y\equiv X+Y \pmod{ (X,Y)^2}$, 
\item $[ \lambda ]_{\Phi_d} X\equiv \lambda X \pmod{X^2}$ pour $\lambda \in \OC_K$. 
\end{enumerate}

\begin{defi}
On appelle $\widehat{\MC}^0_{ LT}$ le foncteur qui \`a un objet $A$ dans $\CC$ associe  les doublets $(F, \rho)$ \`a isomorphisme pr\`es o\`u : 
\begin{enumerate}
\item $F$ est un $\OC_K$-module formel sur $A$ de hauteur $d+1$,
\item $\rho : F \otimes A/ \mG_A \to \Phi_d$ est une quasi-isog\'enie.  
\end{enumerate} 
On d\'efinit $\FC^{0,(h)}_{ LT}$ le sous-foncteur de $\widehat{\MC}^0_{ LT}$ des doublets $(F, \rho)$ o\`u $\rho$ est une quasi-isog\'enie de hauteur $h$.   
\end{defi}

\begin{theo}[Drinfeld \cite{dr1} proposition 4.2]
Le foncteur $\FC^{0,(0)}_{ LT}$ est repr\'esentable par $\widehat{\lt}^0=\spf (A_0)$ o\`u $A_0$ est isomorphe \`a $\OC_{\breve{K}} \llbracket T_1, \dots, T_d \rrbracket$.
Le foncteur $\widehat{\MC}^0_{LT}$ se d\'ecompose en l'union disjointe $\coprod_h \FC^{0,(h)}_{LT}$, chaque $ \FC^{0,(h)}_{LT}$ étant isomorphe non-canoniquement à $ \FC^{0,(0)}_{LT}$.  
\end{theo}

D\'efinissons maintenant les structures de niveau. 

\begin{defi}
Soit $n$ un entier sup\'erieur ou \'egal \`a 1. Soit $F$ un $\OC_K$-module formel sur $A$ de hauteur $d+1$, une $(\varpi)^n$-structure de niveau est un morphisme de $\OC_K$-modules formels $\alpha : ((\varpi)^{-n}/ \OC_K)^{d+1} \to F \otimes \mG_A$ qui v\'erifie la condition : 
\[ \prod_{x \in ((\varpi)^{-n}/ \OC_K)^{d+1}} (X - \alpha(x)) \; |  \; [ \varpi^n ](X) \]
dans $A \llbracket X \rrbracket$.   
Si $(e_i)_{ 0 \le i  \le d}$ est une base de $((\varpi)^{-n}/ \OC_K)^{d+1}$, le $(d+1)$-uplet $(\alpha(e_i))_i$ est appel\'e syst\`eme de param\`etres formels. 

On note $\widehat{\MC}^n_{LT}$ le foncteur classifiant, pour tout objet $A$ de $\CC$, les triplets $(F, \rho, \alpha)$ o\`u $(F, \rho) \in \widehat{\MC}^0_{LT}(A)$ et $\alpha$ est une $(\varpi)^n$-structure de niveau. On d\'efinit de m\^eme par restriction, $\FC^{n,(h)}_{LT}$.
\end{defi} 

\begin{theo}[Drinfeld \cite{dr1} proposition 4.3]
\label{theoreplt}
\begin{enumerate}
\item Le foncteur $\FC^{n,(0)}_{LT}$ est repr\'esentable par $\widehat{\lt}^n=\spf (A_n)$ o\`u $A_n$ est local de dimension $d+1$, r\'egulier sur $A_0$. Le morphisme $A_0 \to A_n$ est fini et plat de degr\'e $\card(\gln_{d+1}(\OC_K/ \varpi^n \OC_K))= q^{(d+1)^2(n-1)} \prod_{i=0}^d (q^{d+1}-q^i)$. 
\item Si $(F^{univ}, \rho^{univ}, \alpha^{univ})$ est le groupe formel universel muni de la structure de niveau universelle, tout syst\`eme de param\`etres formels $(x_i)_i$   engendre topologiquement $A_n$ i.e. on a une surjection : 
\begin{align*}
\OC_{\breve{K}} \llbracket X_0, \dots, X_d \rrbracket & \to A_n \\
X_i & \mapsto x_i
\end{align*}
\item L'analytification  $\lt^n=\widehat{\lt}^{n,rig}$ est lisse sur $\breve{K}$ et le morphisme $\lt^n\to \lt^0$ est un revêtement étale  de groupe de Galois $\gln_{d+1} (\OC_K/\varpi^n \OC_K)=\gln_{d+1} (\OC_K)/(1+\varpi^n {\rm M}_{d+1}(\OC_K))$.
\end{enumerate}
\end{theo}

\section{\'Equation des rev\^etements et géometrie de la fibre spécial de $\widehat{\lt}^1=Z_0$  \label{sssectionltneq}} 

D'apr\`es le th\'eor\`eme \ref{theoreplt}, on a une suite exacte : 
\[ 0 \to I_n \to \OC_{\breve{K}} \llbracket X_0, \dots, X_d \rrbracket \to A_n \to 0   \]
associ\'ee \`a un syst\`eme de param\`etres formels $x=(x_0, \dots, x_d)$. 
Nous allons tenter de d\'ecrire explicitement l'id\'eal $I_n$.  D'après la suite exacte précédente,  on a une immersion fermée $Z_0\rightarrow \spf  \OC_{\breve{K}} \llbracket X_0, \dots, X_d \rrbracket $ de codimension $1$ entre deux schémas réguliers.  L'idéal $I_n$ est donc principal.  Pour obtenir un générateur, il suffit d'exhiber un élément de $I_n\backslash \mG^2$ avec  $\mG=(\varpi, X_0,\cdots, X_d)$ l'id\'eal maximal de $\OC_{\breve{K}} \llbracket X_0 , \dots, X_d \rrbracket$.   Nous écrivons aussi  $\bar{\mG}=(X_0,\cdots, X_d)$ l'idéal maximal de $\bar{\F} \llbracket X_0 , \dots, X_d \rrbracket$ et $\mG_{A_n}$ l'id\'eal maximal de $A_n$.   

Inspirons-nous de \cite[3.1]{yosh}.   Si $z=(z_0, \dots , z_m)$ est un $(m+1)$-uplet de points de $\varpi^n$-torsion pour $F^{univ}$ et $a \in (\OC_K / \varpi^n \OC_K)^{m+1}$ (par exemple $m=d$), on note \[l_{a, F}(z)=[\tilde{a}_0]_{F^{univ}} (z_0)+_{F^{univ}}\cdots +_{F^{univ}} [\tilde{a}_m]_{F^{univ}} (z_m)\] où $\tilde{a}_i$ est un relevé de $a_i$ dans $\OC_K$ pour tout $i$. Par d\'efinition de la structure de niveau, on a la relation
\[ [ \varpi^k](T)= U_k(T) \prod_{a \in (\OC_K/ \varpi^k \OC_K)^{d+1}} (T- l_{a, F}(x)) \]
pour $k\leq n$,  o\`u $U_k$ est une unit\'e telle que $U_k(0) \in 1+ \mG_{A_n}$.  En comparant les termes constants de $[\varpi^n](T)/ [\varpi^{n-1}](T)$,  on obtient 
\[ \varpi = (-1)^{q^{n-1}(q-1)}U_n(0)/U_{n-1}(0) \prod_{a \in (\OC_K/ \varpi^n \OC_K)^{d+1} \backslash (\varpi \OC_K/ \varpi^n \OC_K)^{d+1} )} l_{a, F}(x)=:P. \]
Comme la fl\`eche $\mG \to \mG_{A_n}$ est surjective, on peut relever $U_n(0)/U_{n-1}(0)$ en un \'el\'ement $\tilde{U}(X_0, \dots, X_d)$ de $1+\mG$ et $l_{a,F} (x)$ en une série $l_{a,F}(X_0,\ldots, X_d)$ dans $\OC_{\breve{K}} \llbracket X_0 , \ldots, X_d \rrbracket$.  Par construction, $\varpi-P$ est un élément de $I_n$.   Pour prouver qu'il  n'est pas dans $\mG^2$, on vérifie les congruences suivantes  
\begin{prop}[\cite{yosh} Proposition 3.4]
\label{PropConglaF}
\begin{enumerate}
\item Pour tout $a\in(\OC_K/\varpi^n \OC_K)^{d+1}$, on a \[l_{a,F} (X)\equiv l_{a}(X) \pmod{(\varpi^n,X_0,\cdots,X_d)^2}\] où $l_{a}(X)={a}_0 X_0+\cdots + {a}_d X_d$. 
\item Soit $a=ca'\in (\OC_K/\varpi^n \OC_K)^{d+1}$ avec $c$ une unité de $\OC_K/\varpi^n \OC_K$, alors il existe une unité $u_c$ de $\OC_{\breve{K}} \llbracket X_0 , \dots, X_d \rrbracket$ telle que $l_{a,F}(X)=u_c(X) l_{a',F}(X)$.
\item S'il existe $j\leq d+1$ tel que $a_{j+1}=\cdots=a_d=0$ (la condition devient vide si $j=d+1$) alors $l_{a,F}(X)\in   \OC_{\breve{K}}\llbracket X_0,\ldots, X_j\rrbracket$.
\end{enumerate}
\end{prop}

Le théorème suivant s'en déduit

\begin{theo}
\label{theolteq}
On a $I_n=(\varpi-P(X))$. Dit autrement, \[A_n=\OC_{\breve{K}} \llbracket X_0 , \dots, X_d \rrbracket/(\varpi-P(X)).\] Ainsi, $\lt^n$ s'identifie  à l'hypersurface de la boule unité rigide ouverte de dimension $d+1$ d'équation $\varpi=P(X)$.
\end{theo}




\section{Généralités sur les éclatements}

Nous commençons par donner quelques faits généraux sur les éclatements. Nous pourrons ainsi construire des modèles convenables de $\lt^1$. 

\begin{defi}
Soit $X$ un schéma algébrique, $\If \subset \Of_X$ un faisceau d'idéaux et $Y\to X$ l'immersion fermée associée.      L'éclatement $\bl_Y (X)$ (ou $\tilde{X}$) de $X$ le long de $Y$  est  l'espace propre sur $X$
\[
p:  \underline{\proj}_X (\oplus_n \If^n) \rightarrow X.
\]
 Le fermé $E:=p^{-1}(Y)=V(\If \Of_{\tilde{X}})=  V^+(\oplus_{n}  \If^{n+1})$ est   appelé le diviseur exceptionnel.  
\end{defi}

\begin{defi}\label{defitrans}
Reprenons les notations précédentes, si $Z\to X$ est une immersion fermée, on note 
\[
\tilde{Z}=\begin{cases}
p^{-1}(Z) & \si Z\subset Y, \\
\overline{p^{-1}(Z\backslash Y)} & \sinon 
\end{cases} 
\] 
Dans le deuxième cas  $\tilde{Z}$ est appelé le transformé stricte de $Z$. 
\end{defi}

Nous énonçons les résultats principaux sur les éclatements que nous utiliserons: 

\begin{theo}
\label{theoBlowUp}
Soit $X$ un schéma algébrique, $Y\to X$ et $Z\to X$ deux immersions fermées, $E\subset \bl_Y(X)=\tilde{X}$ le diviseur exceptionnel. Les points suivants sont vérifiés: 
\begin{enumerate}

\item (Propriété universelle)  $E\rightarrow \tilde{X}$ est un diviseur de Cartier (i.e. $\If \Of_{\tilde{X}}$ est localement libre de rang $1$). Si $	f:T\to X$ est un morphisme tel que $f^{-1}(Y)$ est un diviseur de Cartier, il existe un unique morphisme sur $X$ 
\[
g: T\rightarrow \tilde{X}
\]
tel que $g^{-1}(E)=f^{-1}(Y)$.

\item (Compatibilité avec les immersions fermées) Supposons   que $Z=V(\Jf) \nsubseteq Y$. On a un isomorphisme canonique $\bl_{Z\cap Y}(Z)\cong  \tilde{Z}\subset \tilde{X}$. En particulier, $\tilde{Z}=V^+(\bigoplus_n( \Jf\cap \If^n))$

\item (Compatibilité avec les morphismes plats) Si $f:T\rightarrow X$ est un morphisme plat, alors on a un isomorphisme canonique $\bl_{f^{-1}(Y)} (T) \cong \bl_{Y}(X)\times_{X,f} T$.  En particulier, si $T=X\backslash Y$ la projection $p$ induit un isomorphisme $p^{-1}(X\backslash Y)\rightarrow X\backslash Y$.

\item (Éclatement le long d'une immersion régulière) Si $Y\rightarrow X$ est une immersion régulière alors 
  $p|_E: E\rightarrow Y$ est une fibration  localement triviale en  espace projectif de dimension $d-1$ où $d=\dim X -\dim Y$. De plus,  si $Z$ est irréductible alors $\tilde{Z}$ l'est aussi.

\end{enumerate}

\end{theo}

\begin{proof}
Le point 1. est prouvé dans  \cite[\href{https://stacks.math.columbia.edu/tag/0806}{Tag 0806}]{stp} (voir \href{https://stacks.math.columbia.edu/tag/0805}{Tag 0805} pour le point 3. et \href{https://stacks.math.columbia.edu/tag/080D}{Tag 080D}, \href{https://stacks.math.columbia.edu/tag/0806}{Tag 080E} pour le point 2.).  Supposons maintenant que  $Y\rightarrow X$ est  une immersion régulière. On a 
\[
E= \underline{\proj}_X(\Of_{\tilde{X}}/ \If \Of_{\tilde{X}})= \underline{\proj}_X(\bigoplus_{n} \If^n/ \If^{n+1})=\underline{\proj}_{X}(\underline{\sym}_{X} \Nf_{Y/X}^\vee )
\]
avec $\Nf_{Y/X}^\vee$ le faisceau conormal de $Y$ sur $X$. L'hypothèse de régularité entraîne le caractère localement libre de  $\Nf_{Y/X}^\vee$ ce qui montre que $p|_E:E\to Y$ est une fibration localement triviale en  espace projectif de dimension $d-1$.   

Prenons  $Z\subset X$ un fermé irréductible.  Si $Z\subset Y$ alors $p^{-1}(Z)\rightarrow Z$ est une fibration localement triviale dont la base et la fibre sont irréductibles. Ainsi $\tilde{Z}=p^{-1}(Z)$ est irréductible.  Si $Z\nsubseteq Y$ alors $\tilde{Z}$ est la clôture de $p^{-1}(Z\backslash Y)\cong Z\backslash Y$ (d'après le point 2.).  Mais $Z\backslash Y$ est    irréductible  en tant qu'ouvert de $Z$,  d'où l'irréductibilité de $\tilde{Z}$. 
\end{proof}

\begin{rem}\label{rembladm}
Si $X$ est une schéma formel  localement noethérien muni de la topologie $\If$-adique, et $Z\subset X$ est un fermé de la fibre spéciale, on peut définir l'éclatement admissible:
\[
\bl_Z(X)= \varinjlim_k \bl_{Z\times_X V(\If^k)}(V(\If^k))
\]
où $V(\If^k)$ est le schéma fermé définit par $\If^k$.
\end{rem}

\section{Transformée stricte et régularité}

Le but de cette section est de comprendre comment se comporte les immersions fermées régulières $Z_1\to Z_2$ lorsque l'on prend la transformée stricte de deux fermés $Z_1$, $Z_2$ d'un espace $X$ que l'on éclate. Le théorème suivant donne des critères précis pour assurer que l'immersion fermée obtenue entre les transformées $\tilde{Z}_1\to \tilde{Z}_2$ reste encore régulière. Il y est donné aussi le calcul de la transformée d'une intersection de fermé dans des cas particuliers. 

\begin{theo}\label{theoblreg}

Donnons-nous un schéma algébrique $X$ et des fermés $Y,Z_1,Z_2\subset X$ et écrivons $p:\tilde{X}=\bl_Y (X)\to X$ (en particulier, $\tilde{Y}$ est le diviseur exceptionnel). On a les points suivants : 

\begin{enumerate}
\item Si on a  $Y\subset Z_1$, $Z_1\subset Z_2$, $Z_2\subset X$ et si ces inclusions  sont des immersions régulières, alors les immersions $\tilde{Y}\cap \tilde{Z}_1\subset \tilde{Z}_1$, $\tilde{Z}_1\subset \tilde{Z}_2$ sont régulières de codimension $1$ pour la première et $\codim_{Z_2}Z_1 $ pour la seconde. En particulier, la flèche $\tilde{Z}_1\subset \tilde{X}$ est régulière en considérant le cas $\tilde{Z}_2=\tilde{X}$.
\item Si on a les immersions régulières suivantes $Y\subset Z_1\cap Z_2$, $Z_1\cap Z_2\subset Z_i$, $Z_i\subset X$ alors \[\tilde{Z}_1\cap \tilde{Z}_2=\begin{cases}
\emptyset &\si Z= Z_1\cap Z_2,\\
\widetilde{Z_1\cap Z_2} &\sinon
\end{cases}\]
\item Si $Z_1$ et $Y$ sont transverses et si les inclusions  $Z_1\cap Y\subset Y$, $Y\subset X$  sont des immersions régulières, alors $\tilde{Z}_1=p^{-1}(Z_1)$ et $\tilde{Z}_1\cap \tilde{Y}\subset \tilde{Y}$ est régulière de codimension $\codim_Y (Z_1\cap Y)$. 
\item Si $Z_1$ et $Y$ sont transverses et $Y\subset  Z_2$, $Z_2\subset X$ sont des immersions régulières, alors $\tilde{Z}_1\cap \tilde{Y}$ et $\tilde{Z}_2$  sont transverses.
\end{enumerate}

\end{theo}

\begin{rem}\label{remsuitreg}

Les hypothèses des points précédents peuvent se réécrire en termes de suites régulières. Par exemple, les deux premiers points supposent localement l'existence d'une suite régulière $(x_1,\cdots,x_s)$ et de deux sous-ensembles $S_1, S_2 \subset \left\llbracket 1,s\right\rrbracket$ (avec $S_2\subset S_1$ pour le premier point) tel que $Y=V(x_1,\cdots,x_s)$, $Z_i=V((x_j)_{j\in S_i})$ pour $i=1,2$. Pour le troisième point, on demande localement en plus de la suite régulière $(x_1,\cdots,x_s)$ l'existence d'une partition $\left\llbracket 1,s\right\rrbracket=T\amalg S$ tel que $Y=V((x_j)_{j\in T})$ et $Z_1=V((x_j)_{j\in S})$ et le dernier point est une synthèse des cas précédents. Nous laissons au lecteur le soin de trouver des interprétations similaires aux conclusions de l'énoncé en termes de suites régulières locales. 

\end{rem}
\begin{proof}

On remarque que les conclusions du théorème peuvent se vérifier localement. De plus, on a pour tout ouvert $U=\spec (A)\subset X$ affine d'après \ref{theoBlowUp} 2., 3.:
\[
p^{-1}(U)=\bl_{U\cap Y} U \et  \tilde{Z}_i\cap p^{-1}(U) =\bl_{Z_i\cap U\cap Y} Z_i\cap U =\widetilde{Z_i\cap U},
\]
on peut se ramener à étudier les objets qui vont suivre quand  $X=U=\spec (A)$ et $Y\subset U$. Donnons-nous une suite régulière $(x_1,\cdots, x_s)$ dans $A$, et introduisons  pour tout sous-ensemble $S\subset \left\llbracket 1,s\right\rrbracket$, des idéaux $I_S =\sum_{j\in S} x_j A$ et des fermés $Z_S=V(I_S)$. On fixe $S_0$ et on pose $Y=Z_{S_0}$, $p:\tilde{X}=\bl_{Y}(X)\to X$. On construit comme dans \ref{defitrans} les transformées strictes $\tilde{Y}_S$ de $Y_S$ (pour $S_0\nsubseteq S$) et on note $\tilde{\If}_S\subset \Of_{\tilde{X}}$ les idéaux associés. D'après la Remarque \ref{remsuitreg}, il suffit de prouver le résultat local suivant : 

\begin{lem}
\label{lemReg2}
En reprenant les notations précédentes $X=\spec A$, $Y$, $Z_S $, $p:\tilde{X}\to X$, $\tilde{Y}$, $\tilde{Z}_S$,  on a les points suivantes 

\begin{enumerate}

\item Si $S_1$ est une partie de $S_0$ et $n \geq 1$, on a $I_{S_1}\cap I_{S_0}^n= I_{S_1} I_{S_0}^{n-1}$. En particulier, si $S_1$, $S_2$ sont  deux parties de $S_0$, on a \[\tilde{Z}_{S_1}\cap \tilde{Z}_{S_0}=\begin{cases}
\emptyset &\si S_0= S_1\cup S_2,\\
\tilde{Z}_{S_1\cup S_0} &\sinon
\end{cases}\]

\item Si $S_1$ est une partie de $\left\llbracket 1,s\right\rrbracket$ disjointe de $S_0$ et $n \geq 1$, on a $I_{S_1}\cap I_{S_0}^n=I_{S_1} I_{S_0}^n $. Dans ce cas, on a  $ \tilde{Z}_{S_1}= p^{-1}(Z_{S_1})$.

\item Il existe un recouvrement affine de $\tilde{X}= \bigcup_{i\in S_0} U_i$ tel que pour tout $i$, il existe   une suite   régulière $(\tilde{x}_1^{(i)},  \ldots,   \tilde{x}_s^{(i)})\in \Of(U_i)$ vérifiant 

\begin{itemize}

\item[•] $\widetilde{Z_{S_1}}\cap U_i =\emptyset$ si $i\in S_1\subsetneq S_0$. 

\item[•] $ V((\tilde{x}_j)_{j\in S_1})=\tilde{Z}_{S_1}\cap U_i$ si $S_1\subset S_0\backslash\{i\}$ ou si $S_1\cap S_0 =\emptyset$.  

\item[•]   $\tilde{Y}\cap U_i= V(\tilde{x}_i)$.

\end{itemize}

\item  Si $S_1$ est une partie stricte  de $S_0$ et $S_2$ est disjointe de $S_0$, on a $\tilde{Z}_{S_1}\cap \tilde{Z}_{S_2}$ (resp. $\tilde{Z}_{S_0}\cap \tilde{Z}_{S_1}\cap \tilde{Z}_{S_2}$) est de codimension $|S_1|+|S_2|$ (resp.  $|S_1|+|S_2|+1$). 
\end{enumerate}

\end{lem}
\end{proof}

\begin{proof}

Commençons par prouver pour tout $n>1$, les égalités entre idéaux suivantes pour $S_1\subset\left\llbracket 1,s\right\rrbracket$  :
\[I_{S_1}\cap I_{S_0}^n=\begin{cases} I_{S_1} I_{S_0}^{n-1} &\si S_1\subset S_0\\
I_{S_1} I_{S_0}^n & \si S_1\cap S_0 =\emptyset
\end{cases}\]

Nous commençons par cette observation utile qui permettra de réduire la preuve du résultat au cas $n=1$.
\begin{claim}\label{claimreg}

Reprenons la suite régulière $(x_1,\cdots,x_s)$ dans $A$, pour  $T_1\subset T_2\subset\left\llbracket 1,s\right\rrbracket$ et n'importe quel idéal $J$, on a l'identité suivante :
\[I_{T_1}J \cap I_{T_2}^n=I_{T_1} (J\cap I_{T_2}^{n-1})\]

\end{claim}

\begin{proof}

Prenons un élément $x\in I_{T_2}$  et écrivons-le sous la forme :
\[
y=\sum_{t\in T_2} x_t a_t.
\]
L'hypothèse de régularité de la suite $(x_1,\cdots,x_s)$ entraîne que l'élément $y$ est dans $I^n_{T_2} $ si et seulement si chaque terme $a_t$ est dans $ I^{n-1}_{T_2}$ (dit autrement $I^{}_{T_2}/I^{n}_{T_2}\cong \bigoplus_{t\in T_2}x_{t} A/I^{n-1}_{T_2}$). En particulier, si  chaque $a_t$ est dans $J$ (c'est-à-dire $y\in I_{T_2}J$) et $y\in I_{T_2}^n$ alors  nous avons $a_t\in J\cap I_{T_2}^{n-1}$ ce qui montre l'inclusion $(I_{T_2}^{n-1} J )\cap I^n_{T_2} \subset I_{T_2}(J\cap I_{T_2}^{n-1})$ quand $T_1=T_2$. Celle dans l'autre sens étant triviale, on en déduit le résultat dans ce cas.

Si, de plus, $a_t=0$ quand $t\notin T_1$ (ie. $y\in (I_{T_1}J)\cap  I_{T_2}^n$) alors $y\in I_{T_1}(J\cap I_{T_2}^{n-1})$) grâce au raisonnement précédent. Le cas plus général s'en déduit.

\end{proof}
Utilisons cette observation pour prouver les égalités précédentes  et supposons qu'elle sont vraies au rang $n$ pour $n\ge 1$ (ainsi qu'au rang 1). On a alors pour $S_1\subset \left\llbracket 1,s\right\rrbracket$ 
\[I_{S_1}\cap I_{S_0}^{n+1}=I_{S_1}\cap I_{S_0}^{n}\cap I_{S_0}^{n+1}= (I_{S_1} I_{S_0}^{n-1})\cap I_{S_0}^{n+1}=I_{S_1} (I_{S_0}^{n-1}\cap I_{S_0}^{n})=I_{S_1} I_{S_0}^{n}\]
si $S_1\subset S_0$ et
\[I_{S_2}\cap I_{S_0}^{n+1}=I_{S_2}\cap I_{S_0}\cap I_{S_0}^{n+1}=(I_{S_0} I_{S_2} )\cap I_{S_0}^{n+1}=I_{S_0} (I_{S_2}\cap I_{S_0}^{n})=I_{S_2} I_{S_0}^{n+1}\] si $S_1\cap S_0 =\emptyset$. On en déduit le résultat au rang $n+1$.

Nous nous sommes ramenés au cas $n=1$ où on a clairement $I_{S_1}\cap I_{S_0}=I_{S_1}$ si $S_1\subset S_0$ ce qui prouve la première équation. Quand $S_1\cap S_0 =\emptyset$, on raisonne par récurrence sur $|S_1|$. Si $|S_1|=0$, on pose $I_{S_1}=0$ et le résultat est trivial.  Supposons le résultat vrai pour pour les parties strictes de $S_1\neq\emptyset$ et fixons $j_0\in I_{S_1}$.  Par hypothèse de récurrence, on a  $I_{S_1\backslash\{j_0\}}\cap I_{S_0}=I_{S_1\backslash\{j_0\}} I_{S_0}$. Il suffit de  montrer que  la flèche naturelle $ I_{S_1}I_{S_0}/I_{S_1\backslash\{j_0\}}I_{S_0} \to I_{S_1}\cap I_{S_0}/I_{S_1\backslash\{j_0\}}\cap I_{S_0}$ est un isomorphisme par théorème de Noether.  Observons le diagramme de $A$-modules commutatif suivant  
\[
\begin{tikzcd}
I_{S_1}I_{S_0} \ar[r] & I_{S_1}\cap I_{S_0} \\   I_{S_0} \ar[r] \ar[u, "\times x_{j_0}"]&  I_{S_0} \ar[u, "\times x_{j_0}"']
\end{tikzcd}
\]
et montrons qu'il induit un diagramme commutatif (nous allons justifier que les flèches verticales sont bien définies) 
\[
\begin{tikzcd}
I_{S_1}I_{S_0}/ I_{S_1\backslash\{j_0\}}I_{S_0}  \ar[r] & (I_{S_1}\cap I_{S_0})/( I_{S_1\backslash\{j_0\}}\cap I_{S_0} )\\ 
I_{S_0}/( I_{S_0}\cap I_{S_1\backslash\{j_0\}}) \ar[r] \ar[u, "\times x_{j_0}"]& (I_{S_0}+I_{S_1\backslash\{j_0\}})/ I_{S_1\backslash\{j_0\}} \ar[u, "\times x_{j_0}"']
\end{tikzcd}
\]
La flèche horizontale inférieure est clairement bijective, et nous allons utiliser la régularité de $x_r$ pour montrer que les flèches verticales sont bien définies et sont des isomorphismes.

   Par définition de $I_{S_1\backslash\{j_0\}}$,  la flèche  $b\in I_{S_0} \mapsto x_{j_0} b\in I_{S_1}I_{S_0}/I_{S_1\backslash\{j_0\}}I_{S_0}$ est surjective. Si $b$ est dans le noyau, $x_{j_0} b\in I_{S_1\backslash\{j_0\}}$ d'où $b\in I_{S_1\backslash\{j_0\}}\cap I_{S_0}=I_{S_1\backslash\{j_0\}}I_{S_0}$ par régularité et hypothèse de récurrence.   Ainsi la multiplication par $x_{j_0}$ induit un isomorphisme $I_{S_0}/I_{S_1\backslash\{j_0\}} I_{S_0} \iso I_{S_1}I_{S_0} /I_{S_1\backslash\{j_0\}}I_{S_0}$. 

Étudions   $M:=(I_{S_1}\cap I_{S_0})/( I_{S_1\backslash\{j_0\}}\cap I_{S_0} )$.  On observe le diagramme commutatif dont les deux lignes horizontales sont exactes
\[
\begin{tikzcd}
0 \ar[r] & I_{S_1}\cap I_{S_0} \ar[r] &   I_{S_1} \oplus I_{S_0} \ar[r] & I_{S_1}+I_{S_0}  \ar[r] & 0 \\
0 \ar[r] & I_{S_1\backslash\{j_0\}}\cap I_{S_0} \ar[r] \ar[u, hook] &   I_{S_1\backslash\{j_0\}} \oplus I_{S_0} \ar[r] \ar[u, "\iota_1", hook] & I_{S_1\backslash\{j_0\}}+I_{S_0}  \ar[r]  \ar[u, "\iota_2", hook]& 0 
\end{tikzcd}
\]
Par régularité, on  a $\coker \iota_1 = I_{S_1}/I_{S_1\backslash\{j_0\}}=x_{j_0} (A /I_{S_1\backslash\{j_0\}}) $ et $\coker \iota_2\cong  x_{j_0} (A/(I_{S_1\backslash\{j_0\}}+I_{S_0}))$.  Comme les flèches verticales sont injectives,  \[M\iso\ker( A/I_{S_1\backslash\{j_0\}}\fln{}{} A/(I_{S_1\backslash\{j_0\}}+I_{S_0}))= (I_{S_1\backslash\{j_0\}}+I_{S_0})/I_{S_1\backslash\{j_0\}}\] ce qui termine l'argument par récurrence.

Maintenant,  terminons la  preuve de 1. et 2. grâce aux égalités précédentes. Le résultat pour 2. est clair, passons à 1. Pour cela, raisonnons dans un cadre un peu plus général. Pour $J_1$, $J_2$, $J_3$ des idéaux d'un anneaux $B$, on a toujours $(J_1+J_2)J_3=J_1J_3+J_2J_3$ mais rarement $(J_1+J_2)\cap J_3=J_1\cap J_3+J_2\cap J_3$ sauf si on a par exemple un idéal ${J'_3}$ tel que  $(J_1+J_2)\cap J_3=(J_1+J_2) J'_3$ et $J_i\cap J_3=J_i J'_3$ pour $i=1,2$.  Un raisonnement similaire au cas général précédent  permet d'établir si  \[\bigoplus_n I_{S_1\cup S_2}\cap I_{S_0}^n=\bigoplus_n (I_{S_1}+I_{S_2}) I_{S_0}^{n-1}=\bigoplus_n I_{S_1} I_{S_0}^{n-1}+\bigoplus_n I_{S_2} I_{S_0}^{n-1}=\bigoplus_n I_{S_1}\cap I_{S_0}^{n}+\bigoplus_n I_{S_2}\cap I_{S_0}^{n}\] quand $S_1\cup S_2\subset S_0$. En terme de fermé de $\tilde{X}$, cela se traduit par $ \tilde{Y}_{S_1\cup S_2}=\tilde{Y}_{S_1}\cap  \tilde{Y}_{S_2}$ si $S_0\neq S_1\cup S_2$. Le cas  $S_0= S_1\cup S_2$ sera montré plus tard (pas de risque d'argument circulaire). 

Passons  au point 3.  Notons $\tilde{A}=A\oplus I \oplus I^2\oplus\cdots$ et si $x\in I^n$,  on note $x^{[n]}\in \tilde{A}$ l'élément $x$ vu comme un élément homogène de degré $n$.   Considérons le recouvrement $\tilde{X}=\bigcup_{i\in S_0} U_i$ avec $U_i=D^+(x_i^{[1]})$. On a  \[\widetilde{V(x_j)}=V^+(\bigoplus_n x_jI^{n-1})= V^+(x_j^{[0]}\tilde{A}+ x_j^{[1]} \tilde{A} )\] si  $j\in S_0$ d'après le point $1$ et \[\widetilde{V(x_j)}=V^+(\bigoplus_n x_jI^{n})= V^+(x_j^{[0]}\tilde{A})\] si $j\notin S_0$ d'après le point $2$.  Ainsi, d'après l'égalité dans $\tilde{A}$: $x_j^{[0]}x_i^{[1]}-x_j^{[1]}x_i^{[0]}$, on a :
\begin{itemize}
\item[•]$\widetilde{V(x_i)}\cap U_i=\emptyset$, 
\item[•]$\tilde{Y}_{S_0} =V(x_i^{[0]})$
\item[•]$ \widetilde{V(x_j)}\cap U_i= V\left(\frac{x_j^{[1]}}{x_i^{[1]}}\right) $  pour $j\in S_0\backslash\{i\}$,
\item[•]$ \tilde{Y}_j\cap U_i= V\left({x_j^{[0]}}\right)\cap U_i$ si $j\notin S_0$.
\end{itemize} 
Posons alors $\tilde{x}_j^{(i)}=\frac{x_j^{[1]}}{x_i^{[1]}}$ si $j\in S_0\backslash\{i\}$,  $\tilde{x}_j^{(i)}=x_j^{[0]}$ si $j\notin S_0$ et $\tilde{x}_i^{(i)}=x_i^{[0]}$. On a par construction \[\tilde{Y}_{\{j\}}\cap U_i= V\left({x_j^{(i)}}\right)\] pour $1\le j\le s$. De plus, on a grâce aux points 1. et 2. pour  $S_1\subsetneq S_0\backslash\{i\}$ ou $S_1\cap S_0=\emptyset$ \[\tilde{Y}_{S_1}=\bigcap_{j\in S_1} \tilde{Y}_{\{j\}}=V\left({x_j^{(i)}}\right)_{j\in S_1}\] et \[\tilde{Y}_{S_0}\cap \tilde{Y}_{S_1}=V\left({x_j^{(i)}}\right)_{j\in S_1\cup \{i\}}.\]

Il reste à prouver  que $(\tilde{x}^{(i)}_j)_{j\in \left\llbracket 1,s\right\rrbracket}$ est régulière pour montrer qu'il s'agit bien de la suite recherchée. Commençons par montrer que la sous-suite $(\tilde{x}^{(i)}_j)_{j\in S_0}$ est régulière  par récurrence sur $|S_0|$.  Quand $S_0=\{s_0\}$, on a $X=\tilde{X}$ et $\tilde{x}^{(s_0)}_{s_0}=x_{s_0}$ qui est régulier par hypothèse. Supposons le résultat vrai pour $|S_0|-1\geq 1$. On se place en $U_i\subset \tilde{X}$. 
Comme  $x_{j_0}$ ($j_0\neq i$) n'est pas un diviseur de $0$ dans $A$,  donc $x_{j_0}^{[1]}$ ne l'est pas non plus dans $\tilde{A}$ et $\tilde{x}_{j_0}^{(i)}=\frac{x_{j_0}^{[1]}}{x_i^{[1]}}$ est un élément régulier dans $\Of(U_i)$. On veut montrer que $(\tilde{x}_{j}^{(i)})_{j\in S_0\backslash\{j_0\}}$ est régulier dans $\Of(U_s)/(\tilde{x}_{j_0})=\Of(U_i\cap \widetilde{V(x_{j_0})})$. Mais 
\[
\widetilde{V(x_{j_0})}=\bl_{V((\tilde{x}_{j}^{(i)})_{j\in S_0\backslash\{j_0\}})}(V(x_{j_0}))=:\proj(\widetilde{A/ x_{j_0}})
\]      
et $U_i\cap  \widetilde{V(x_{j_0})} =D^+(x_i^{[1]}\widetilde{A/ x_{j_0}}) $  est un ouvert standard. On conclut alors  par hypothèse de récurrence sur $\widetilde{V(x_{j_0})}$. 

Maintenant, nous devons montrer que la suite $(x_j^{(i)})_{j\notin S_0}$ est régulière dans $\Of(U_i)/(x_j^{(i)})_{j\in S_0}=A/I_{S_0}$ et cela découle de la régularité de  $(x_j)_{j\in \left\llbracket 1,s\right\rrbracket}$ dans $A$.

Pour 4., on a d'après le point précédent pour $S_1$ est une partie stricte  de $S_0$ et $S_2$ est disjointe de $S_0$, on a \[\tilde{Y}_{S_1}\cap \tilde{Y}_{S_2}=V\left({x_j^{(i)}}\right)_{j\in S_1\cup S_2}\]
et
\[\tilde{Y}_{S_0}\cap\tilde{Y}_{S_1}\cap \tilde{Y}_{S_2}=V\left({x_j^{(i)}}\right)_{j\in S_1\cup S_2\cup\{i\}}. \]
Par régularité de $(\tilde{x}^{(i)}_j)_{j\in \left\llbracket 1,s\right\rrbracket}$, on voit que la codimension de $\tilde{Y}_{S_1}\cap \tilde{Y}_{S_2}$ (resp. $\tilde{Y}_{S_0}\cap\tilde{Y}_{S_1}\cap \tilde{Y}_{S_2}$) est $|S_1\cup S_2|=|S_1|+| S_2|$ (resp. $|S_1\cup S_2\cup\{i\}|=|S_1|+| S_2|+1$).
 \end{proof}
\section{Construction de modèles de $\lt^1$}

Rappelons que nous avons construit une tour de revêtement $(\lt^n)_n$ de la boule unitée rigide. Dans toute la suite, on prend $n=1$ et nous construirons une suite de modèles $Z_0=\widehat{\lt}^1,Z_1,\cdots,Z_d$ de $\lt^1$. Nous avons vu que le modèle $Z_0$ était $\spf (A_1)$ avec $A_1$ local et régulier sur $\OC_{\breve{K}}$. Il est important de préciser que nous munissons $A_1$ de la topologie $p$-adique et non de la topologie engendrée par l'idéal maximal. En particulier, la fibre spéciale de $Z_0$ est de la forme  
\[
Z_{0,s}= \spec ( \bar{\F} \llbracket X_0,\ldots, X_d  \rrbracket  )/ (\prod_{a\in \F^{d+1}\backslash\{0\}} l_{a,F}(X)).
\] 
 où on note encore $l_{a,F}$ la réduction modulo $\varpi$. On a alors une décomposition  $Z_{0,s}=\bigcup_a Y_a$ où $Y_a=V(l_{a,F})$ \cite[Définition 3.7 + sous-section 3.2., page 11]{yosh}. On a le résultat suivant : 
\begin{prop}[\cite{yosh} Proposition 3.9]\label{propltfibrspe}
\begin{enumerate}

\item Si $S\subset \F_q^{d+1}\backslash\{0\}$ et $M= \langle S \rangle^{\perp}$, le fermé $\bigcap_{a\in S} Y_a$ ne depends que de $M$ et pas de $S$ et nous le noterons  $Y_M$.

\item Si $S$ est minimal (i.e.  est une famille libre),  alors la suite $(l_{a,F})_{a\in S}$ est    regulière.  

\item Les composantes irréductibles de $Z_{0,s}$ sont les fermés $Y_M$ où $M=a^\perp$ est un hyperplan. En particulier, chaques composantes irréductibles a multiplicité $q-1$ qui est le cardinal d'une $\F$-droite de $\F^{d+1}$ à laquelle on a retiré l'élément nul.

\item L'application $M\mapsto Y_M$ est une bijection croissante $G^{\circ}$-équivariante entre l'ensemble des sous-espaces de $\F^{d+1}$ et l'ensemble des intersections finies de composantes irréductibles de $Z_{0,s}$. De plus, chaque $Y_M$ est irreductible. 


\end{enumerate}

\end{prop}

\begin{proof}

Il s'agit d'une application de \ref{PropConglaF} qui a été faite dans \cite[Proposition 3.9.,Lemma 3.11.]{yosh}.

\end{proof}

On définit 
\[
Y^{[h]}=\uni{Y_N}{N\subset \F^{d+1}\\ {\rm dim}(N)=h}{}.
\]

 Nous allons définir une suite d'espaces $Z_0,\ldots, Z_d$  sur $\OC_{\breve{K}}$, et des fermés $Y^{[0]}_i\subset \cdots \subset Y_i^{[d]}\subset Z_{i,s}$ tels que $Y_0^{[h]}=Y^{[h]}$  s'inscrivant dans le diagramme commutatif 
\[\xymatrix{ 
 Z_i\ar[r]^{\tilde{p}_i} \ar@/_0.75cm/[rrrr]_{p_i} & Z_{i-1}\ar[r]^{\tilde{p}_{i-1}} & \cdots \ar[r]^{\tilde{p}_2} & Z_1 \ar[r]^{\tilde{p}_1} & Z_0}\]
avec $\tilde{p}_i$ un éclatement .  Supposons $Z_0, \ldots, Z_{i}$, et $Y_i^{[0]}, \ldots, Y_i^{[d]}$ préalablement construit,   on définit $Z_{i+1}$ et $Y_{i+1}^{[h]}$ via le relation de récurrence 
\[
Z_{i+1}=\bl_{Y^{[i]}_h}(Z_{i})  \et Y_{i+1}^{[h]}= \widetilde{Y^{[h]}_i}. 
\]
De même, pour tout fermé $Y\subset Z_{0,s}$, on définit pour tout $i$ un  fermé $Y_i\subset Z_{i,s}$ via $ Y_0=Y $ et $ Y_{i+1}= \widetilde{Y_i}$. En particulier, nous pourrons nous intéresser à la famille des fermés $Y_{M,i},Y_{a,i}\subset Z_{i,s}$ pour $M\nsubseteq \F^{d+1}$ et $a\in \F^{d+1}\backslash\{0\}$.   

\begin{rem}

Notons que les éclatements considérés ici sont bien admissibles au sens de \ref{rembladm} car les fermés que nous avons considérés sont contenus dans la fibre spéciale (ce qui explique le choix de la topologie dans la définition de $Z_0$). Ce fait justifie aussi que chacun des espaces $Z_i$ obtenus sont encore des modèles de $\lt^1$.

Nous pouvons aussi donner la construction de ces modèles dans un cadre plus "algébrique". En effet, on peut construire dans $\spec (A_1)$ un analogue de la stratification $(Y^{[h]})_h$ et considérer les éclatements successifs suivant les fermés de cette stratification. On obtient alors une chaîne d'éclatement qui fournit encore, lorsqu'on complète $p$-adiquement, les modèles $Z_i$ décrits auparavant.

\end{rem}

L'interprétation modulaire de $Z_0$ fournit une action naturelle des trois groupes $\OC_D^*$ (qui s'identifient aux isogénies de $\Phi^d$), $G^{\circ}$ (en permutant les structures de niveau) et de $W_K$ (sur $Z_0\otimes \OC_C$ pour ce dernier). Comme les morphismes de vari\'et\'es envoient les irr\'eductibles sur les irr\'eductibles, tout groupe agit en permutant les $Y_a$ et donc les $Y_M$ de même dimension. Ainsi, $Y^{[h]}_0$ est stable sous les actions de $G^\circ$, $\OC_D^*$ et $W_K$ et $Z_1$ hérite d'une action de ces groupes qui laisse stable $Y^{[h]}_1$ par propriété universelle de l'éclatement. Par récurrence immédiate, on a encore une action de $G^\circ$, $\OC_D^*$ et $W_K$ sur $Z_i$ qui laisse stable $Y^{[h]}_i$ pour tout $i$. 

Dans le cas particulier où $i=1$, $Y^{[0]}_0= Y_{\{0\}}$ est l'unique point fermé de $Z_0$.  Le diviseur exceptionnel $Y_{\{0\},1}$ s'identifie à $\P^d_{\overline{\F}}$ et hérite des actions de $G^\circ, \OC_D^*$ et $W_K$.

\section{Les composantes irréductibles de la fibre spéciale des modèles $Z_i$}

Nous souhaitons décrire les composantes irréductibles de la fibre spéciale de chacun des modèles intermédiaires $Z_i$. Si $Y$ est une composante irréductible de $Z_{i,s}$, on écrit $Y^{lisse}=Y\backslash \bigcup Y'$ où $Y'$ parcourt l'ensemble    des composantes irréductibles de $Z_{i,s}$ différentes de $Y$. Le but de cette section est de démontrer le théorème suivant : 
\begin{theo}
\label{TheoZH}
Soit $0\leq i \leq d$ un entier, On a 

\begin{enumerate}
\item Les composantes irréductibles de la fibre spéciale de $Z_i$ sont les fermés  de dimension $d$ suivants   $(Y_{M,i})_{M: \dim M\in \left\llbracket 0,i-1\right\rrbracket\cup\{d\}}$. 

\item Les intersections non-vides  de composantes irréductibles de $Z_i$ sont en bijection avec les drapeaux $M_1\subset \cdots \subset M_k$  tels que $\dim M_{k-1}< i$ via l'application \[M_1\subset \cdots \subset M_k \mapsto \bigcap_{1\leq j \leq k} Y_{M_j,i}.\]

\item  Si $Y_{M,i}$ est une composante irréductible de $Z_{i,s}$ avec $\dim M\neq d$, alors les morphismes naturels $\tilde{p}_h$ avec $j=i+1,\ldots,  d$ induisent  des isomorphismes $  Y_{M,d}^{lisse}  \cong\cdots \cong Y_{M,i+1}^{lisse} \cong Y_{M,i}^{lisse}$. 

\item Le changement de base du tube $]Y_{\{0\},d}^{lisse}[_{Z_d} \otimes \breve{K}(\varpi_N)\subset {\rm LT}_1 \otimes \breve{K}(\varpi_N)$ admet un modèle lisse dont la fibre spéciale est isomorphe à la variété de Deligne-Lusztig $\dl_{\bar{\F}}$ (ici, $N=q^{d+1}-1$ et $\varpi_N$ est le choix d'une racine $N$-ième de $\varpi$).
\end{enumerate}

\end{theo}

En fait, nous allons prouver le résultat plus technique et plus précis.

\begin{lem}\label{lemirrzh}

Soit $0\leq i \leq d$ un entier fixé, la propriété $\Pf_i$ suivante est vérifiée :

\begin{enumerate}

\item Pour tout sous-espace vectoriel $M\subsetneq \F^{d+1}$, $Y_{M,i}$ est irréductible et  $Y_{M,i}\subset Z_i$ est une immersion régulière de codimension $1$ si $\dim M<i$ et $\codim M$ sinon.

\item Si $M=\langle S \rangle^{\perp}$ avec $ S=\{a_1,\ldots, a_s\}\subset \F^{d+1}$, on a $Y_{a_1,i}\cap \cdots \cap Y_{a_s,i}= \begin{cases}  Y_{M,i}, &  \si \dim M\geq i \\ 
\emptyset  & \sinon \end{cases} $  

\item Si $M\subsetneq N\subset \F^{d+1}$ avec $\dim M \geq i$,  alors $Y_{M,i}\to Y_{N,i}$ (bien définie d'après le point précédent),  est une  immersion régulière de codimension  $\dim N-\dim M$.

\item Si $M_1\subset \cdots \subset M_k$ avec $\dim M_{k-1}< i$, $\bigcap_{1\leq j \leq k} Y_{M_j,i}\subset Y_{M_k,i}$ est une  immersion régulière de codimension $s-1$. En particulier, $Y_{M_k,i}$ et $\bigcap_{1\leq j \leq k-1} Y_{M_j,i}$ sont transverses.

\item Si $i> i'=\dim M$,  on a un isomorphisme $Y_{M,i}=  \tilde{p}_{i,i'}^{-1}(Y_{M,i'})$. 

\end{enumerate}   

\end{lem}

\begin{proof}
Raisonnons par récurrence sur $i$.
Quand $i=0$,  $\Pf_0$ correspond à la Proposition \ref{propltfibrspe} (4. et 5. sont vides dans ce cas). Supposons $\Pf_i$ vrai et montrons $\Pf_{i+1}$.   On sait ($\Pf_i$ 1. et 2.) que $Y_{N,i}\to Z_i$ est régulière pour tout $N$ et $Y_{N_1,i}\cap Y_{N_2,i}=\emptyset$ si $\dim N_1=\dim N_2=i$ et $N_1\neq N_2$.  Donc $Y^{[i]}_i= \bigsqcup_N Y_{N,i}\rightarrow Z_i$ est régulière. Considérons le  recouvrement  $Z_i =\bigcup_{N:\dim N=i} U_N$ avec $U_N:=Z_i\backslash (\bigcup_{N'} Y_{N',i})$ où $N'$ parcours les espaces vectoriels de dimension $i$ différents de $N$. On a par construction $U_N\cap Y^{[i]}_i=Y_{N,i}$ et les résultats \ref{theoBlowUp} 2. et 3. entraînent
\begin{align*}
\tilde{p}_{i+1}^{-1}(U_N)&=\bl_{Y_{N,i}} (U_N)\\
Y_{M,i+1} \cap \tilde{p}_{i+1}^{-1}(U_N)&=\bl_{Y_{M,i}\cap Y_{N,i}} (Y_{M,i}\cap U_N)=\widetilde{Y_{M,i}\cap U_N}
\end{align*} 
et on peut appliquer  \ref{theoBlowUp} 4.,  \ref{theoblreg}  pour montrer que $Y_{M,i+1} \cap \tilde{p}_{i+1}^{-1}(U_N)$ vérifie les propriétés voulues. Mais comme ces dernières sont des égalités entre des intersections et des propriétés de régularité qui sont locales, on en déduit la propriété $\Pf_{i+1}$ ce qui termine la preuve.  

\end{proof}

\begin{proof}[Démonstration du Théorème \ref{TheoZH}]
Pour le premier point, tous les fermés de la forme $Y_{M,i}$ avec $\dim M\in \left\llbracket 0,i-1\right\rrbracket\cup\{d\}$ sont irréductibles d'après le résultat précédent. On a aussi montré qu'ils étaient tous de dimension $d$ et que les  intersections de deux éléments étaient ou vide ou de dimension $d-1$. On en déduit qu'il n'existe aucune relation d'inclusion entre deux de ces fermés. Il suffit donc de prouver que les fermés étudiés recouvrent bien  la fibre spéciale de  $Z_i$. Pour cela, raisonnons par récurrence sur $i$. Au rang $0$, cela provient  de \ref{propltfibrspe}. Supposons le résultat vrai au rang $i$. Dans ce cas, l'union $\bigcup_{\dim M\in \left\llbracket 0,i-1\right\rrbracket\cup\{d\}} Y_{M,i}$ contient $Z_i\backslash Y^{[i]}_i$ et on en déduit d'après \ref{theoBlowUp} 2.
\[
Z_{i+1}\backslash Y^{[i]}_{i+1}\subset \bigcup_{\dim M\in \left\llbracket 0,i-1\right\rrbracket\cup\{d\}} Y_{M,i+1}.
\]
De plus, le diviseur exceptionnel $Y^{[i]}_{i+1}$ est l'union des fermés de la forme $Y_{M,i+1}$ avec $\dim M =i$ ce qui prouve que la famille exhibée correspond bien à la décomposition parties irréductibles de la fibre spéciale de $Z_{i+1}$.

Pour le deuxième point, les intersections de composantes irréductibles sont de la forme $\bigcap_{1\leq j \leq k-1} Y_{M_j,i}\cap \bigcap_{a\in S}Y_{a,i}$ avec $S\subset \F^{d+1}\backslash\{0\}$ et $\dim M_j<i$. Posons $M_k=\langle S \rangle^{\perp}$, l'intersection précédente est vide si $\dim M_k <i$. Supposons que ce ne soit pas le cas,  il s'agit de voir que $\bigcap_{1\leq j \leq k} Y_{M_j,i}$ est  non vide si et seulement si $(M_j)_j$ est un drapeau. Si c'en est un, d'après \ref{theoblreg},  l'intersection est de codimension 
\[
\begin{cases}
s &\si \dim M_{k}< i\\
s-1 + \codim M_k &\sinon 
\end{cases}
\]
qui est  inférieure ou égale à $d$  ce qui montre que ces intersections sont non-vides. Pour l'autre sens, supposons  $M_{j_1}\nsubseteq M_{j_2}$ et $M_{j_1}\nsubseteq M_{j_2}$ pour $j_1\neq j_2$, nous devons montrer  $Y_{M_{j_1},i}\cap Y_{M_{j_2},i}=\emptyset$. On peut trouver sous ces hypothèses  $i'$ tel que \[\dim M_{j_1}\cap M_{j_2}< i' \le \min (\dim M_{j_1},\dim M_{j_2}) <i.\] On a alors $Y_{M_{j_1},i'}\cap Y_{M_{j_2},i'}=\emptyset$ d'après $\Pf_{i'}$ 2. De plus, si $i_2>\dim M_{j_1}\cap M_{j_2}$ et $Y_{M_{j_1},i_2}\cap Y_{M_{j_2},i_2}=\emptyset$, alors $Y_{M_{j_1},i_2+1}\subset\tilde{p}^{-1}_{i_2+1}(Y_{M_{j_1},i_2})$ (idem pour $Y_{M_{j_2},i_2}$) d'où \[Y_{M_{j_1},i_2+1}\cap Y_{M_{j_2},i_2+1}\subset\tilde{p}^{-1}_{i_2+1}(Y_{M_{j_1},i_2}\cap Y_{M_{j_2},i_2})=\emptyset\] ce qui montre par récurrence $Y_{M_{j_1},i}\cap Y_{M_{j_2},i}=\emptyset$.

Passons à 3.  Si   $Y$ est irréductible sur $Z_{i,s}$ et n'est pas de la forme $Y_{a,i}$ pour $a\in \P^{d}(\F)$, alors $Y^{lisse}$  ne rencontre aucun $Y_{a,i}$ et donc aucun $Y_{M,i}$ avec $\dim M=i$ car ce sont des intersections  de tels fermés de la forme $Y_{a,i}$ d'après  1. En particulier, $Y^{lisse}\cap Y^{[i]}_{i}=\emptyset$.  De même,  $(\tilde{Y})^{lisse}$ ne rencontre pas $ Y^{[i]}_{i+1}$ qui est une union de composantes irréductibles toujours d'après 1. Il reste à prouver que $(\tilde{Y})^{lisse}=\tilde{p}_{i+1}^{-1}(Y^{lisse})$.  Cela découle de l'isomorphisme (cf  \ref{theoBlowUp} 3.) pour $Y'$ une composante irréductible différente de $Y$ :
\[
\tilde{p}_{i+1}: \widetilde{Y'}\backslash  Y^{[i]}_{i+1} \iso  Y'\backslash  Y^{[i]}_{i}.
\]

Pour le dernier point, cela a été montré dans  \cite[Proposition 5.2]{yosh} quand $i=0$. Pour les autres valeurs de $i$, il suffit d'appliquer le point 3.
\end{proof}




\section{Semi-stabilité du modèle $Z_d$}

Notre but dans cette section est de rappeler le théorème principal de \cite[Théorème 4.2]{yosh}.  Pour le confort du lecteur, nous donnerons les grandes lignes de la preuve.  L'énoncé est le suivant 

\begin{theo}[Yoshida]
Pour tout point fermé $x\in Z_d$,   l'anneau local complété au point $x$ est isomorphe à 
\[
\OC_{\breve{K}} \llbracket T_0,\ldots, T_d  \rrbracket /(T_0^{e_0}\cdots T_r^{e_r}- \varpi )
\]
avec $r\leq d$ et $e_i$ premier à $p$. 
\end{theo}

\begin{rem}

Expliquons succinctement la stratégie de la preuve. L'idée est de calculer par récurrence l'anneau local complété en les points fermés des différents  éclatements $Z_{i}$.  Plus précisément, on montre une décomposition pour cet anneau en un point $z$ sous la forme  $\OC_{\breve{K}}\llbracket X_0,\ldots, X_d \rrbracket/(\prod_j f_j^{m_j}-\varpi)$ où  les fermés $(V(f_j))_j$ décrivent  l'ensemble des composantes irréductibles rencontrant le point $z$ considéré et l'entier $m_i$ est la multiplicité de la composante associée. D'après le théorème   \ref{TheoZH} 2.,  ces composantes sont de la forme $Y_{M,i}$ avec $\dim {M}< i$ ou $Y_{a,i}$ avec $a\in \F^{d+1}\backslash\{0\}$.       \'Etant donnée une  famille des composantes $(Y_{M_j,i})_{j} \cup (Y_{a_j,i})_j$, les  fonctions associées $(f_{M_j})_j\cup f_{a_j})_{j}$  forment  une suite régulière  si et seulement si les $(a_j)_j$ sont  libres (et donc une famille $\bar{\F}$-libre dans l'espace cotangent).   On peut alors  estimer le nombre des composantes de la forme $Y_{a,i}$ et montrer qu'il  y en a au plus une quand $i=d$.  Dans ce cas, les $(f_j)$ forment une suite régulière vérifiant la propriété de liberté précédente et  la semi-stabilité  en découle.  

Pour pouvoir décrire les fonctions $f_j$ et les $m_j$, on raisonne par récurrence sur $i$.   Quand $i=0$, il s'agit du théorème  \ref{theolteq}.  Si $z$ est  un point fermé de $Z_{i+1}$, son image par l'éclatement  $\tilde{z}\in Z_{i}$ est un point fermé  et $\widehat{\Of}_z$ se voit comme l'anneau local complété en un point du diviseur exceptionnel  d'un éclatement de $\widehat{\Of}_{\tilde{z}}$.  Grâce au lemme \ref{lemReg2}, on peut décrire explicitement le lien entre les composantes irréductibles de $\widehat{\Of}_{\tilde{z}}$ et celles de $\widehat{\Of}_{z}$. 

\end{rem}

\begin{proof}

Fixons $z\in Z_i$ pour $i\leq d$. Calculons  l'anneau local complété dans $z$.    Comme $p_i: Z_{i}\rightarrow Z_0$ est propre, $p_i(z)$ est un point fermé et est donc  $Y_{\{0\},0}$ par localité de $A_0=\Of(Z_0)$. Ainsi, on a $z\in p_i^{-1}(Y_{\{0\},0})=Y_{\{0\},i}$ d'après le lemme  \ref{lemirrzh} 5.      D'après le théorème  \ref{TheoZH} 2.,  il existe un drapeau  $M_0 \subsetneq M_1\subsetneq \cdots \subsetneq M_{j_0} \subsetneq M_{j_0+1} $    avec $\dim M_{j_0} <i $  tel que les composantes irréductibles rencontrant $z$ sont les $Y_{M_0,i}, \cdots,Y_{M_{j_0,i}}$     ainsi que les fermés $Y_{a,i}$ avec $a^\perp\supset M_{j_0+1} $.   De plus,  $M_0=\{0\}$ d'après la discussion précédente.
Quitte à translater par un élément de $\gln_{d+1}(\F)$, on peut supposer que $M_j=\langle e_0,\ldots, e_{d_{j}} \rangle$ avec $(e_i)_i$ la base canonique de $\F^{d+1}$ (en particulier, $d_i=\dim M_i-1$).     Nous voulons prouver par récurrence sur $i$ que l'anneau local complété $\widehat{\Of}_z$ au point $z$ est isomorphe à 
\[
\OC_{\breve{K}} \llbracket X_0,\ldots, X_d \rrbracket /(uX_{d_0}^{m_0} \cdots X_{d_{j_0}}^{m_{j_0}}(\prod_{a\in M_{j_0+1}^{\perp}\backslash\{0 \}  }  f_a) - \varpi)
\]
avec

\begin{enumerate}

\item  $u$  est une unité et     $m_j=|M_j^{\perp}  |-1=q^{d-d_j}-1$. 

\item $V(f_a)= \spec \widehat{\Of}_z  \times_{Z_i} Y_{a,i}$ et $V(X_{d_j})=  \spec \widehat{\Of}_z  \times_{Z_i} Y_{M_j,i}$.  

\item   Si $a=\sum_{j\geq k }  a_j e_j$ avec $k> d_{j_0+1}$,  on peut trouver des relevés $\tilde{a}_j\in \OC_{\breve{K}}$ tel que  
\[
f_a= \sum_{i} \tilde{a}_j X_j \mod  (X_{k}, \ldots , X_d)^2.\]

\end{enumerate} 

 Cette description de l'anneau local complété quand  $i=d$ établi le théorème.  En effet, on a     $M_{j_0+1}^{\perp}=0$ et le produit $(\prod_{a\in M_{j_0+1}^{\perp}\backslash\{0 \}  }  f_a)$ est vide.   De plus,  d'après le lemme de Hensel,   $\widehat{\Of}_{z}^*$ est $n$-divisible pour $n$ premier à $p$.  En posant $\tilde{X}_{0}= u^{1/m_0} X_{0}$ et $\tilde{X}_j=X_j$ pour $j\neq 0$, on voit 	que $\widehat{\Of}_z$ est de la forme voulue.

Quand $i=0$, voyons que la description précédente de l'anneau local complété découle de \ref{propltfibrspe}. Le fermé $Y_{\{0\}}$ est l'unique point fermé de $Z_{0}$ et l'anneau local complété est 
\[
\Of(Z_0)= \OC_{\breve{K}}\llbracket X_0,\ldots, X_d \rrbracket/(u\prod_{a\in \F^{d+1}\backslash \{0\}} l_{a,F} -\varpi)
\]
avec $u$ une unité.   Décrivons les quantités $j_0, (M_j)_{0\leq j\leq j_0+1} $ et $(f_a)_{a\in M_{j_0+1}^{\perp}\backslash \{0\}}$ pour $\Of(Z_0)$.  On a $j_0={-1}$, $M_{j_0+1}=M_{0}=\{0\}$, en particulier  $M_0^\perp= \F^{d+1}$ .  Posons  $f_{a}=l_{a,F}$ pour $a\in \F^{d+1}\backslash \{0\}$ et on obtient la formule voulue grâce à \ref{PropConglaF} et l'action transitive de $\gln_{d+1}(\F)$ sur la famille $(l_{a,F})$.   

Supposons le résultat pour $i\geq 0$, montrons-le pour $i+1$ et prenons $z\in Z_{i+1}$ un point fermé.  Par propreté des éclatements,  $\tilde{z}={p}_{i+1}(z)$ est un point fermé de $Z_{i}$.  Si $\dim M_{j_0} < i$, alors $\tilde{z}$ n'est pas dans le centre de l'éclatement et $p_{i+1}$ induit un isomorphisme entre $\hat{\Of}_{\tilde{z}}$ et $\hat{\Of}_z$.  Le résultat dans ce cas s'en déduit par hypothèse de récurrence.

Sinon, toujours par hypothèse de récurrence,   on écrit l'anneau local complété sous la  forme  
\[
\widehat{\Of}_{\tilde{z}}\cong \OC_{\breve{K}} \llbracket \tilde{X}_0,\ldots, \tilde{X}_d \rrbracket /(u\tilde{X}_{d_0}^{m_{0}} \cdots \tilde{X}_{d_{j_0-1}}^{m_{j_0-1}}(\prod_{a\in M_{j_0}^{\perp}\backslash\{0 \}  }  f_a) - \varpi)
\]
où $j_0-1$, $(M_j)_{j\le j_0}$, et $(f_a)_a$ vérifient les propriétés escomptées au rang.  Pour relier l'anneau $\hat{\Of}_{\tilde{z}}$  à l'anneau $\hat{\Of}_z$, on observe (cf. théorème \ref{theoBlowUp} 3.) l'identité   


\[
S: = \bl_{Y_{M_{j_0+1},i}\times_{{Z_i}} {\widehat{Z}_{i, \tilde{z}}}    }(\widehat{Z}_{i, \tilde{z}} )=  Z_{i+1}\times_{Z_{i}}\widehat{Z}_{i, \tilde{z}}. 
\]
avec $\widehat{Z}_{i, z}=\spec \widehat{\Of}_{\tilde{z}}$ car $\widehat{Z}_{i,\tilde{z}}\rightarrow Z_{i}$ est plat.  

D'après 2. et 3. de  l'hypothèse  de récurrence et \ref{TheoZH} 3.,  $(\tilde{X}_j)_{j=i,\dots, d}$ est une suite régulière  et 
\[
S= \bl_{V((f_{e_j})_j)}(\widehat{Z}_{i,\tilde{z}})= \proj( \widehat{\Of}_{\tilde{z}}[T_i, \ldots, T_{d}]/(T_{j_1} \tilde{X}_{j_2}- T_{j_2} \tilde{X}_{j_1})).  
\] 
L'anneau $\spf \hat{\Of}_z$ correspond au complété en un point fermé  du diviseur exceptionnel $\proj( \overline{\F} [T_i,\ldots, T_d])$, et ces derniers  sont en bijection avec $\P^{d-i}(\overline{\F})$.  Mais par hypothèse, $z\notin Y_{e_{i},i+1}$  et le point fermé en question est dans $D^+(T_i)= \spec B$.    
Écrivons-le sous la forme \[z=[1,z_{i+1}, \cdots, z_d]=[1,z_{i+1}, \cdots, z_{d_{j_0+1}-1}, 0, \cdots, 0]\in \P^{d-i}(\bar{\F}).\]  On a alors  
\[
B:= \OC_{\breve{K}}\llbracket \tilde{X}_0,\ldots, \tilde{X}_d \rrbracket [\tilde{T}_{i+1}, \ldots, \tilde{T}_{d}]/(\tilde{T}_{j}\tilde{X}_i-\tilde{X}_j, P-\varpi) 
\] 
avec $\tilde{T}_j=\frac{T_j}{T_i}$ et l'idéal associé à $z$ est  \[\mG_{z}=\mG_{\tilde{z}}B+(\tilde{T}_j-\tilde{z}_j)_j,\] où  $\tilde{z}_j\in \OC_{\breve{K}}$ est un relèvement  de $z_j\in \overline{\F}$,  et  $P=u\tilde{X}_{d_0}^{m_0} \cdots \tilde{X}_{d_{j_0-1}}^{m_{j_0-1}}(\prod_{a\in M_{j_0}^{\perp}\backslash\{0 \}  }  f_a) $.  

Par noetherianité,  on obtient  d'après \cite[Prop 10.13]{atimcdo}
\[
\widehat{\Of}_{z}= \OC_{\breve{K}}\llbracket   X_0, \cdots, X_d  \rrbracket/(P(X_0, \cdots,X_i, (X_{i+1}+\tilde{z}_{i+1}) X_i,\cdots, (X_d+\tilde{z}_d) X_i)-\varpi)
\] en posant\footnote{Notons que l'on a pris $\tilde{z}_j=0$ si $j \geq d_{j_0+1}$} \[
X_j= \begin{cases}
\tilde{X}_j & \si j\leq i \\
\tilde{T}_j-\tilde{z}_j & \sinon. 
\end{cases}
\]  
Comme $f_{a}(\tilde{X}_{0},\dots, \tilde{X}_d)\in (\tilde{X}_{i},\dots, \tilde{X}_d)$  d'après 3.,  alors 
\[
f_a( X_0, \cdots,X_i, (X_{i+1}+\tilde{z}_{i+1}) X_i,\cdots, (X_d+\tilde{z}_d) X_i)=: X_i g_a({X}_{0},\dots, {X}_d).
\]
Ainsi, réécrivons $P$ sous la forme 
\begin{eqnarray*}
P(X_0, \cdots,X_i, (X_{i+1}+\tilde{z}_{i+1}) X_i,\cdots, (X_d+\tilde{z}_d) X_i) & = & u X_{d_0}^{m_{0}} \cdots X_{d_{j_0-1}}^{m_{j_0-1}} \prod_{a \in M_{j_0}^{\perp}\{ 0 \} }  X_i g_a  \\ 
	& = & v X^{m_0}_{d_0}\cdots X^{m_{j_0-1}}_{d_{j_0-1}} X_{d_{j_0}}^{m_{j_0}} \prod_{a\in M_{j_0+1}^{\perp}\backslash  \{0\} } g_a \\ 
\end{eqnarray*}
avec $v=u \prod_{a\notin M_{j_0+1}^{\perp}\backslash\{0\}} g_{a}$.    D'après \ref{lemReg2} 3.,  chacun des $g_a$ dans le produit définissant $v$ est inversible car $z\notin Y_{a,i}$. On obtient 
\[
\widehat{\Of}_{z}=\OC_{\breve{K}}\llbracket X_0, \ldots,  X_d \rrbracket/((v X_{d_0}^{m_0} \cdots  X_{d_{j_0}}^{m_{j_0}} (\prod_{a\in M_{j_0+1}^{\perp}\backslash\{0 \}  }  g_a) - \varpi ).
\]


Montrons que cette description vérifie les hypothèses demandées.  Pour le point 1., c'est clair par construction.   Le point 3. montre la régularité de $\tilde{X}_{d_{0}}, \ldots,  \tilde{X}_{d_{j_0}}, \tilde{X}_{d_{j_0+1}+1}, \tilde{X}_{d_{j_0+1}+2},\ldots, \tilde{X}_d$.  \footnote{on peut aussi raisonner par récurrence sur $i$ grâce à \ref{lemReg2}.}   Ainsi, on peut appliquer les arguments de la preuve de  \ref{lemReg2}  pour obtenir 
\[
V(X_{d_j}) = Y_{M_j,i+1}\times_{Z_{i+1}} \widehat{Z}_{i+1,z} \et  V(g_{a}) = Y_{a,i+1}\times_{Z_{i+1}} \widehat{Z}_{i+1,z}
\]
pour $j\leq j_{0}+1$ et $a\in M_{j_0+1}^{\perp} \backslash \{0\}$.     Le point 2. est alors vérifié.  Pour le point 3. en $z$, cela découle de l'hypothèse analogue en $\tilde{z}$ et de la relation \[g_a({X}_{0},\dots, \tilde{X}_d)=f_a( X_0, \cdots,X_i, (X_{i+1}+\tilde{z}_{i+1}) X_i,\cdots, (X_d+\tilde{z}_d) X_i)/X_i\] pour $a\in M_{j_0+1}^{\perp}\backslash \{0\}$.


\end{proof}

\section{Relation avec un modèle semi-stable d'une variété de Shimura\label{sssectionlt1sh}}

Nous rappelons la construction de modèles entiers pour certaines variétés de Shimura développés dans (\cite[chapitre III.3]{yosh}) et (\cite[chapitre 2.4, chapitre 4]{harrtay}) et nous décrivons le lien entre ces modèles et ceux des sections précédentes. 

Soit $F=EF^+$ un corps CM avec $F^+$totalement réel  de degré\footnote{noté $d$ dans \cite{harrtay}} $k$ et $E$ quadratique imaginaire où $p$ est décomposé. Fixons $r$ un entier positif et donnons nous $w_1(=w),\cdots,w_r$ des places de $F$ au-dessus de $p$ tel que \footnote{Pour montrer l’existence d’une telle extension $F$, on se ramène au cas où $K/\Q_p$ est galoisienne quitte à prendre une clôture galoisienne et à passer aux invariants sous Galois. Par résolubilité de l’extension, on peut raisonner sur des extensions abéliennes cycliques par dévissage. Trouver $F$ revient alors à montrer l'existence de certain  caractère  par théorie du corps de classe ce qui est réalisé dans \cite[appendice A.2]{blggt} par exemple.} $F_w=K$. Soit $B/F$ une algèbre à division de dimension\footnote{noté $n^2$ dans \cite{harrtay}} $(d+1)^2$ déployé en la place $w$ (voir p.51 de \cite{harrtay} pour les hypothèses supplémentaires imposées sur $B$). Nous appelons $G$ le groupe réductif défini dans \cite{harrtay} p.52-54. Soit $U^p$ un sous-groupe ouvert compact assez petit de $G(\A^{\infty,p})$ et, $m=(1,m_2,\cdots, m_r)\in \N^r$, Nous nous intéresserons au problème modulaire considéré dans \cite[p. 108-109]{harrtay} qui est  représentable par  ${X}_{U^p,m}$ un schéma  sur $\OC_K$ propre, plat de dimension  $d+1$. 

Nous n'allons pas décrire en détail le problème modulaire représenté par ${X}_{U^p,m}$ mais nous rappelons seulement que le foncteur associé classifie les quintuplets $(A,\lambda, \iota, \eta^p , (\alpha_i)_i)$ à isomorphisme près où $A$ est un schéma abélien de dimension $k(d+1)^2$ muni d'une polarisation $\lambda$ première à $p$, d'une $\OC_B$-action et  d'une structure de niveau $\alpha_1 :\varpi^{-1}\varepsilon\OC_{B_w}/\varepsilon\OC_{B_w}\to \varepsilon A[\varpi]$ où $\varepsilon$ est un idempotent de $\mat_{d+1}(\OC_K)$ (avec quelques compatibilité entre ces données). On rappelle que  l'on a une identification $\varepsilon\OC_{B_w}\cong \OC_K^{d+1}$ par équivalence de Morita. 
Nous écrirons $(A,\lambda,\iota,\eta^p, (\alpha_i)_i)$ le quintuplet universel du problème modulaire $X_{U^p,m}$.

La fibre spéciale $\bar{X}_{U^p,m}={X}_{U^p,m}\otimes \F$ admet une stratification par des sous-schémas fermés réduits $\bar{X}_{U^p,m}=\bigcup_{0\le h\le d} \bar{X}_{U^p,m}^{[h]}$ de dimension pure $h\in\llbracket 0,d\rrbracket$. L'espace $\bar{X}_{U^p,m}^{[h]}$ est la clôture de l'ensemble des points fermés $s$ où $\GC_{A,s}$ a pour hauteur étale inférieure ou égale à $h$ (cf p. 111 dans \cite[Corollary III.4.4]{harrtay}). Chacun de ces espaces admet un recouvrement $\bar{X}_{U^p,m}^{[h]}=\bigcup_M \bar{X}_{U^p,m,M}$ où $M$ parcourt les sous-espaces $\F$-rationnels de $\P_{\F}^d$ de dimension $h$ \cite[3.2 (see Remark 10(2))]{mant}. Nous allons maintenant étendre les scalaires à $\spec(\OC_{\breve{K}})$ et noter ${{\rm Sh}}_{}:={X}_{U^p,m}\otimes \OC_{\breve{K}}$, $\bar{{\rm Sh}}_{}:=\bar{X}_{U^p,m}\otimes \bar{\F}$, $\bar{{\rm Sh}}_{M}:=\bar{X}_{U^p,m,M}\otimes \bar{\F}$ et $\bar{{\rm Sh}}_{}^{[h]}:=\bar{X}_{U^p,m}^{[h]}\otimes \bar{\F}$ ainsi que $\hat{{\rm Sh}}$ la complétion de ${\rm Sh}$ le long de la fibre spéciale $\bar{{\rm Sh}}$.

Comme dans la section précédente, on construit une suite de modèles entiers $\hat{\rm Sh}_0=\hat{\rm Sh}, \hat{\rm Sh}_1,\cdots,\hat{\rm Sh}_d$ s'inscrivant dans un diagramme \[\xymatrix{ 
 \hat{\rm Sh}_i \ar[r]^{\tilde{p}_i} \ar@/_0.75cm/[rrrr]_{p_i} & \hat{\rm Sh}_{i-1} \ar[r]^{\tilde{p}_{i-1}} & \cdots \ar[r]^{\tilde{p}_2} & \hat{\rm Sh}_1 \ar[r]^{\tilde{p}_1} & \hat{\rm Sh}_0}\] en réalisant successivement les éclatement admissibles le long des  transformés stricts de $\bar{{\rm Sh}}^{[i-1]}$ (cela revient aussi à éclater ${\rm Sh}$ puis à compléter $p$-adiquement chaque modèle intermédiaire). On notera aussi $\bar{{\rm Sh}}_{M,i}$ le transformé strict de $\bar{{\rm Sh}}_{M}$. De même si $\bar{s}$ est un point géométrique fermé centré en $s\in \bar{X}_{U^p,m}^{[0]}$ que l'on voit par changement de base comme un point fermé de $\bar{{\rm Sh}}_{}^{[0]}$, on note $\bar{{\rm Sh}}_{\bar{s},i}:=p_i^{-1}(\bar{s})$. En particulier, $\bar{{\rm Sh}}_{\bar{s},1}$ s'identifie à un espace projectif $\P^d_{\bar{\F}}$. Le théorème suivant relie les constructions relatives au premier revêtement de la tour de Lubin-Tate à celles de la variété de Shimura. 

\begin{theo}[Harris-Taylor,Yoshida]\label{theomodeleshltalg}
Soit $\bar{s}$ un point fermé géométrique centré en $s\in\bar{X}^{[0]}_{U^p,m}$ vu comme un point de ${\rm Sh}$.

\begin{enumerate}
\item (\cite[Lemma III.4.1]{harrtay}) On a un isomorphisme \[Z_0\cong \spec(\hat{\Of}_{{\rm Sh},\bar{s}})\cong \spec(\hat{\Of}_{\hat{\rm Sh},\bar{s}})\] Il en résulte un morphisme $Z_0\to \hat{\rm Sh}$.
\item (\cite[Lemma 4.4]{yosh}) Via cette application $Z_0\to \hat{\rm Sh}$, on a des isomorphismes \[Y_M\cong {\bar{\rm Sh}}_M \times_{\hat{\rm Sh}}Z_0\et Y^{[h]}\cong \bar{{\rm Sh}}^{[h]}\times_{\hat{\rm Sh}} Z_0\]
\item (\cite[Lemma 4.6]{yosh}) Pour tout $i\le d$, \[Z_i\cong {\hat{\rm Sh}}_i\times_{\hat{\rm Sh}} Z_0 \et  Y_{\{0\},i}\cong {\bar{\rm Sh}}_{\bar{s},i}\times_{\bar{\rm Sh}_i} Z_i\] De plus, il existe un voisinage étale de $Z_d$ vu comme un fermé de ${\rm Sh}_d$ de réduction semi-stable. 
\item $\bar{\rm Sh}_{\bar{s},d}$ est une composante irréductible de $\bar{{\rm Sh}}_d$ et les autres composantes rencontrent $\bar{\rm Sh}_{\bar{s},d}$ exactement en les espaces $\bar{\rm Sh}_{\bar{s},d}\cap \bar{\rm Sh}_{M,d}$ où $M$ est un hyperplan $\F$-rationnel de $\P^d_{zar,\bar{\F}}$.
\item Le tube $]\bar{{\rm Sh}}_{\bar{s},d}^{lisse}[_{\hat{{\rm Sh}}_d}\otimes \breve{K}(\varpi_N)\subset \lt^1 \otimes \breve{K}(\varpi_N)$ au dessus du lieu lisse $\bar{\rm Sh}_{\bar{s},d}^{lisse}:=\bar{\rm Sh}_{\bar{s},d}\backslash \bigcup_Y Y$ (où $Y$ parcourt les composantes irréductibles de $\bar{\rm Sh}_d$ différentes de $\bar{{\rm Sh}}_{\bar{s},d}$) admet un modèle lisse isomorphe à la variété de Deligne-Lusztig $\dl^d_{\bar{\F}}$.
\end{enumerate}    

\end{theo}

\begin{proof}
Le premier point a été prouvé dans \cite[Lemma III.4.1]{harrtay}. Le premier isomorphisme du point 2. s'obtient en explicitant l'isomorphisme du point 1. et en comparant la définition de ${\rm Sh}_M$ \cite[Lemma 4.4]{yosh}. L'autre isomorphisme du point 2. et ceux du point 3. s'en déduisent par combinatoire et compatibilité du procédé d'éclatement \cite[Lemma 4.6]{yosh} et de complétion $p$-adique. La dernière assertion du troisième point découle de l'argument technique de la preuve de \cite[Proposition 4.8 (i)]{yosh}. Le point 4. résulte de la description de la fibre spéciale de $Z_d$ réalisée dans   \ref{TheoZH}.  Le dernier  point  découle de  4. et  de \ref{TheoZH} 3. et 4.
\end{proof}

\section{Cohomologie des vari\'et\'es de Deligne-Lusztig\label{ssectiondl}}


Considérons la variété 
 $$\Omega^d_{\F}:=\P_{\F}^d \backslash \bigcup_H H,$$ o\`u $H$ parcourt l'ensemble des hyperplans $\F_q$-rationnels. Elle admet une action naturelle de $\gln_{d+1}(\F)$ et un revêtement fini étale $\gln_{d+1}(\F)$-équivariant \[\dl_{\F}^d := \{x\in \A^{d+1}_{\F_q}\backslash \{0\}  |  \prod_{a\in \F_q^{d+1}\backslash \{0\} } a_0x_0++a_dx_d=(-1)^d  \}\] de groupe de Galois $\F_{q^{d+1}}^*$ via $\zeta\cdot (x_0,\cdots,x_d)=(\zeta x_0,\cdots, \zeta x_d)$.  
 
On note $\dl^d_{\bar{\F}}$ l'extension des scalaires de $\dl_{\F}^d$ \`a $\bar{\F}$. Soit $l\ne p$ un nombre premier, l'intérêt principal de cette construction est l'étude de la partie cuspidale  
de la cohomologie $l$-adique à support compact de $\dl^d_{\bar{\F}_q}$ dont  nous allons rappeler la description. 

 
   Soit $\theta: \F_{q^{d+1}}^*\to \overline{\mathbb{Q}}_l^*$ un caractère. Si $M$ est un $  \overline{\mathbb{Q}}_l[\F_{q^{d+1}}^*]$-module on note 
 $$M[\theta]={\rm Hom}_{\F_{q^{d+1}}^*}(\theta, M).$$
 On dit que le caractère $\theta$ est 
    \emph{primitif} 
  s'il ne se factorise pas par la norme $\F_{q^{d+1}}^*\to \F_{q^e}^*$ pour tout diviseur propre $e$ de $d+1$.

 Si $\pi$ est une représentation de $\gln_{d+1}( \F_q)$, on dit que $\pi$ est \emph{cuspidale} si 
 $\pi^{N(\F_q)}=0$ pour tout radical unipotent $N$ d'un parabolique propre de $\gln_{d+1}$.  
  La théorie de Deligne-Lusztig (ou celle de Green dans notre cas particulier) fournit: 
 
 \begin{theo}\label{DLet}
  Soit $\theta: \F_{q^{d+1}}^*\to \overline{\mathbb{Q}}_l^*$ un caractère.
  
  a) Si $\theta$  est primitif, alors 
  $\hetc{i}(\dl_{\bar{\F}_q}^d, \bar{\Q}_l)$ est nul pour $i\ne d$ et 
  $$\bar{\pi}_{\theta,l}:=\hetc{d}(\dl_{\bar{\F}_q}^d, \bar{\Q}_l)[\theta]$$
  est une $\gln_{d+1}( \F_q)$-représentation irréductible, cuspidale, de dimension  $(q-1)(q^2-1) \dots (q^d-1)$. Toutes les repr\'esentations cuspidales sont ainsi obtenues.

  b) Si $\theta$ n'est pas primitif, aucune repr\'esentation cuspidale n'intervient dans $\oplus_{i}\hetc{i}(\dl_{\bar{\F}_q}^d, \bar{\Q}_l)[\theta]$. 
 \end{theo}

  \begin{proof}
   Voir \cite[cor. 6.3]{DL}, \cite[th. 7.3]{DL}, \cite[prop. 7.4]{DL}, \cite[prop. 8.3]{DL}, \cite[cor. 9.9]{DL}, \cite[Proposition 6.8.(ii) et remarques]{yosh} pour ces résultats classiques. 
  \end{proof}
  
    Ainsi, la partie cuspidale $\hetc{0}(\dl_{\bar{\F}_q}^d, \bar{\Q}_l)_{\rm cusp}$ de $\oplus_{i} \hetc{i}(\dl_{\bar{\F}_q}^d, \bar{\Q}_l)$ est concentrée en degré $d$, où elle est donnée par $\oplus_{\theta} \bar{\pi}_{\theta,l}\otimes\theta$, la somme directe portant sur tous les caractères primitifs.
 
 \begin{rem} (voir \cite[6.3]{DL}) Soit $N=q^{d+1}-1$ et fixons de isomorphismes $\F_{q^{d+1}}^*\simeq \Z/N\Z$ et 
 $\Z/N\Z^{\vee}\simeq \Z/N\Z$.
    Soient $\theta_{j_1}$ et $\theta_{j_2}$ deux caract\`eres primitifs vus comme des \'el\'ements de $\Z/N\Z$ via $j_1$, $j_2$, les repr\'esentations $\bar{\pi}_{\theta_{j_1}}$ et  $\bar{\pi}_{\theta_{j_2}}$ sont isomorphes si et seulement si il existe un entier $n$ tel que $j_1= q^n j_2$ dans $\Z/N\Z$ . 
 \end{rem}



Nous aurons besoin d'un analogue des résultats  précédents pour 
   la cohomologie rigide. Cela a été fait par Grosse-Klönne dans \cite{GK5}. 
   Si $\theta: \F^*_{q^{d+1}}\to \bar{K}^*$ est un caractère, posons 
   $$\bar{\pi}_{\theta}=\hrigc{*}(\dl_{\F_q}^d/ \bar{K})[ \theta]:=\bigoplus_{i}\hrigc{i}(\dl_{\F_q}^d/ \bar{K})[ \theta],$$
 où 
 $$\hrigc{i}(\dl_{\F_q}^d/ \bar{K}):=\hrigc{i}(\dl_{\F_q}^d)\otimes_{W(\F_q)[1/p]} \bar{K}$$
 et où $M[\theta]$ désigne comme avant la composante $\theta$-isotypique de $M$.

\begin{theo}\label{theodlpith}
Fixons un premier $l\ne p$ et un isomorphisme $\bar{K} \cong \bar{\Q}_l$. Si $\theta$ est un caract\`ere primitif, alors  $$\bar{\pi}_{\theta}:=\hrigc{d}(\dl_{\F_q}^d/ \bar{K})[ \theta]$$ est isomorphe en tant que 
$\gln_{d+1}(\F_q)$-module à $\bar{\pi}_{\theta,l}$, en particulier c'est une représentation irréductible cuspidale.


\end{theo}

\begin{proof} Cela se fait en trois étapes, cf. \cite[4.5]{GK5}. Dans un premier temps, on montre \cite[3.1]{GK5} que 
les $\bar{K}[\gln_{d+1}(\F_q) \times \F^*_{q^{d+1}}]$-modules virtuels 
\[ \sum_i (-1)^i \hetc{i}(\dl_{\bar{\F}_q}^d, \bar{\Q}_l) \et \sum_i (-1)^i \hrigc{i}(\dl_{\F_q}^d/ \bar{K}) \]
co\"incident. 
Il s'agit d'une comparaison standard des formules des traces de Lefschetz en
cohomologies étale $l$-adique et rigide. Dans un deuxième temps (et c'est bien la partie délicate du résultat)
on montre que $\bigoplus_{i}\hrigc{i}(\dl_{\F_q}^d/ \bar{K})[ \theta]$ est bien concentré en degré $d$, cf. \cite[th. 2.3]{GK5}. On peut alors conclure en utilisant le théorème \ref{DLet}.
\end{proof}



\section{Automorphismes équivariants des variétés de Deligne-Lusztig}

\begin{lem}\label{lemautdl}
On a 
$\aut_{\gln_{d+1}(\F)}(\Omega^d_{\bar{\F}})= \{ 1 \}$. 
\end{lem}

\begin{proof}
On a $\Of(\Omega^d_{\bar{\F}})= \bar{\F} [X_1, \dots, X_d, \frac{1}{\prod_{a \in \P^d(\F)} l_a(1,X)}]$. Soit $\psi$ un automorphisme $\gln_{d+1}(\F)$-\'equivariant, $\psi$ est d\'etermin\'e par l'image de $X_i$ pour tout $i$ qui sont des \'el\'ements inversibles et qui s'\'ecrivent  en fractions irr\'eductibles de   la forme 
\[ \psi(X_i)=\frac{P_i(X)}{Q_i(X)} = \lambda_i\prod_{a \in \P^d(\F)} l_a(1,X)^{\beta_{a,i}}, \ \ \ \beta_{a,i}\in \Z\]
On veut montrer que $\psi(X_i)=X_i$ pour tout $i$. 
\begin{itemize}
\item Prenons $\sigma : \left\llbracket 1,d \right\rrbracket\to \left\llbracket 1,d \right\rrbracket$ une permutation, on trouve $g_{\sigma} \in \gln_{d+1}(\F)$ tel que $g_{\sigma}.X_i=X_{\sigma(i)}$. La relation $g_{\sigma}.\psi(X_i)=\psi(g_{\sigma}.X_i)$ impose l'\'egalit\'e 
\[ \frac{P_{\sigma(i)}(X)}{Q_{\sigma(i)}(X)}= \frac{P_i(\sigma(X))}{Q_i(\sigma(X))}. \]
Nous nous int\'eressons uniquement \`a $\frac{P_1(X)}{Q_1(X)}$. Nous supposons $d \ge 2$ et nous voulons nous ramener au cas $d=1$.  
\item On consid\`ere $\stab(X_1) \subset \gln_{d+1}(\F)$. Par \'equivariance de $\psi$, on voit que $\frac{P_1}{Q_1}$ est fix\'e par $\stab(X_1)$. Introduisons $U_1= \bigcup_{i \neq 0,1} D^+(z_i^*) \subset \P^d(\F)$ avec $(z_i^*)_{0\le i\le d}$ la base duale de $\F^{d+1}$ (on a identifi\'e $\P^d(\F)$ avec $\P((\F^{d+1})^*)$). Sous ces choix, on a $X_i=\frac{z_i^*}{z_0^*}$ pour $1\le i\le d$. On v\'erifie alors que $l_a(1,X)$ est un polynôme en $X_1$ si et seulement si $a \notin U_1$. On a alors une \'ecriture unique en fractions irr\'eductibles :
\[ \frac{P_1(X)}{Q_1(X)}= \frac{\widetilde{P}_1(X_1)}{\widetilde{Q}_1(X_1)} \prod_{a \in U_1}l_a(1,X)^{\beta_{a,1}}.\]
Comme $\stab(X_1)$ agit transitivement \footnote{Soit $a$, $b$ dans $U_1$. Les familles $\{ (1,0, \dots,0), (0,1,0, \dots, 0), a \}$ et  $\{ (1,0, \dots,0), (0,1,0, \dots, 0), b \}$ sont libres et on peut trouver un endomorphisme qui fixe  $(1,0, \dots, 0)$, $(0,1,0, \dots, 0)$  et qui  envoie $a$ sur $b$.} sur $U_1$ et laisse stable $\frac{\widetilde{P}_1(X_1)}{\widetilde{Q}_1(X_1)}$, on obtient l'\'ecriture 
\[ \frac{P_1(X)}{Q_1(X)}= \frac{\widetilde{P}_1(X_1)}{\widetilde{Q}_1(X_1)} (\prod_{a \in U_1} l_a(1,X))^n \] ie. $\beta_{a,1}$ ne dépend pas de $a$ lorsque $a\in U_1$.
\item Soit $c$ dans $\F^*$, il existe $g_c \in \gln_{d+1}(\F)$ tel que $g_c(X_1)=X_1+c$ et $g_c(X_i)=X_i$ pour $i \neq 1$. $g_c$ laisse stable $U_1$ et on a ainsi, $g_c.(\prod_{a \in U_1} l_a(1,X))^n=(\prod_{a \in U_1} l_a(1,X))^n$. La relation $\psi(g_c.X_1)= g_c.\psi(X_1)$ impose l'\'egalit\'e 
\[ c+ \frac{\widetilde{P}_1(X_1)}{\widetilde{Q}_1(X_1)} (\prod_{a \in U_1} l_a(1,X))^n = \frac{\widetilde{P}_1(X_1+c)}{\widetilde{Q}_1(X_1+c)} (\prod_{a \in U_1} l_a(1,X))^n. \]
D'o\`u 
\[c.( \prod_{a \in U_1} l_a(1,X))^{-n} =\frac{\widetilde{P}_1(X_1+c)}{\widetilde{Q}_1(X_1+c)}-\frac{\widetilde{P}_1(X_1)}{\widetilde{Q}_1(X_1)}  \in \bar{\F}(X_1). \]
 
 Ainsi, $n=0$ car $\prod_{a \in U_1} l_a(1,X) \notin \bar{\F}(X_1)$ et $c\neq 0$. D'où, 
 \[ \psi(X_1) \in \bar{\F}[X_1, \frac{1}{\prod_{a \in \F} (X_1-a)}]^*=\bar{\F}^* \times \prod_{a \in \F} (X_1-a)^{\Z}. \]
 
\item On s'est ramen\'e \`a $d=1$ et on \'ecrit $X=X_1$. On a l'\'ecriture en fractions rationnelles irr\'eductibles $\psi(X)= \frac{\widetilde{P}_1(X)}{\widetilde{Q}_1(X)}$ dont tous les p\^oles et z\'eros sont dans $\F$. De plus,  pour tout $c$, d'apr\`es le point pr\'ec\'edent, on a
\[ c+\frac{\widetilde{P}_1(X)}{\widetilde{Q}_1(X)}= \frac{\widetilde{P}_1(X+c)}{\widetilde{Q}_1(X+c)}. \] 
Comme $c\widetilde{Q}_1+\widetilde{P}_1$ est encore premier à $\widetilde{Q}_1$, on obtient $\widetilde{Q}_1(X+c) = \widetilde{Q}_1(X)$ pour tout $c$ par unicité de l'\'ecriture en fractions rationnelles irr\'eductibles. Ainsi,  $\widetilde{Q}_1(X)= (\prod_{a \in \F}(X-a))^k$ car les racines  de $\widetilde{Q}_1(X)$ vues dans $\bar{\F}$ sont contenues dans $\F$ et $\F$ agit transitivement dessus  par translation. Supposons $k \neq 0$, il existe $c$ tel que $c \widetilde{Q}_1(X) + \widetilde{P}_1(X)$ est non-constant\footnote{Si ce n'est pas le cas, $\widetilde{Q}_1$ et $\widetilde{P}_1$ sont tous deux constants ce qui contredit la bijectivité de $\psi$.} et admet donc  une racine dans $\bar{\F}$ et donc dans $\F$. Cette dernière ne peut être une racine de $\widetilde{Q}_1$ par primalité, ce qui impose $k=0$ et  $\widetilde{Q}_1$ est constante 

\item On a montr\'e que $\psi(X)=P(X) \in \F[X]$. Soit une matrice $g$ telle que $g.X= \frac{1}{X}$, comme $\psi(g.X)=g. \psi(X)$, on a\footnote{Notons que $P(\frac{1}{X})$ signifie que l'on a évalué $P$ en $\frac{1}{X}$ ie. $P(\frac{1}{X})=\sum_i (a_i(\frac{1}{X^i})$ si $P(X)=\sum_i a_i X^i$}
 \[\frac{1}{P(X)}=P(\frac{1}{X}).\]
 Les z\'eros de $P(X)$ sont les p\^oles de $P(\frac{1}{X})$ qui sont réduits au singleton $\{ 0 \}$ d'o\`u $P(X)=\lambda X^n$ pour un certain $n$. Comme $\psi$ est une bijection, $n=1$ et $P=\lambda X$. 

\item D'après ce qui précède, $P(X+c)=P(X)+c$ pour tout $c\in \F$ d'où $\lambda c=c$ et $\lambda=1$. Ainsi, $\psi=\id$!
\end{itemize}
\end{proof}




\section{Cohomologie de De Rham du premier rev\^etement\label{sssectionlt1dr}}

Nous sommes maintenant en mesure d'énoncer et de prouver le résultat technique  principal de cette article.

\begin{theo}\label{theolt1dr}
On a un isomorphisme $\gln_{d+1}(\OC_K)$-équivariant :
\[\hdr{*} (\lt^1/\breve{K}_N)\cong \hrig{*} (\dl^d_{\bar{\F}}/\breve{K}_N)\] Par dualité de Poincaré, on a un isomorphisme semblable pour les cohomologies à support compact.
\end{theo}

\begin{proof}
Comme dans l'énoncé du théorème  \ref{theomodeleshltalg}, on fixe un point géométrique $\bar{s}$ de $\bar{\rm Sh}^{[0]}$ ainsi qu'une identification $Z_0\cong \spf (\hat{\Of}_{{\rm Sh},\bar{s}})$. On se donne de plus un schéma formel p-adique $V$ de réduction semi-stable voisinage  de $Z_d$ dans $\hat{\rm Sh}_d$ cf \ref{theomodeleshltalg} 3. et on note $U^{lisse}$ le modèle lisse de $]\bar{\rm Sh}_{d,\bar{s}}^{lisse}[_U\otimes \breve{K}_N$ construit dans \ref{theomodeleshltalg} 4. et $U^{lisse}_s$ sa fibre spéciale. D'après \ref{theomodeleshltalg} 1., 2., on a une suite d'isomorphismes \[\lt^1\cong]\bar{s}[_{\hat{\rm Sh}}\cong  ]p_d^{-1}(\bar{s})[_{\rm \hat{Sh}_d}  \cong ]\bar{\rm Sh}_{d,\bar{s}}[_{\hat{\rm Sh}_d}\cong ]\bar{\rm Sh}_{d,\bar{s}}[_V\] avec $p_d$ la flèche ${\rm \hat{Sh}_d} \rightarrow \rm{ \hat{Sh} }_0$ obtenue par éclatement.  L'identité $p_d^{-1}(\bar{s})=  ]\bar{\rm Sh}_{d,\bar{s}}[_{\hat{\rm Sh}_d}$ découle de \ref{lemirrzh} 5.  On peut alors appliquer le théorème \ref{theoexcision} dans $V$ qui est $p$-adique de réduction semi-stable et obtenir cette suite d'isomorphismes \[\hdr{*} (\lt^1\otimes \breve{K}_N)\cong \hdr{*} (]\bar{\rm Sh}_{d,\bar{s}}[_V\otimes \breve{K}_N)\cong\hdr{*} (]\bar{\rm Sh}_{d,\bar{s}}^{lisse}[_V \otimes \breve{K}_N)\] Ces morphismes naturels se déduisent d'applications de restriction  qui sont  $\gln_{d+1}(\OC_K)$-équivariantes (l'espace $]\bar{\rm Sh}_{d,\bar{s}}^{lisse}[_V)$ est stable pour cette action).

D'après \ref{theomodeleshltalg} 4. et \ref{theopurete}, on a \[\hdr{*}(]\bar{\rm Sh}_{d,\bar{s}}^{lisse}[_V\otimes \breve{K}_N)\cong\hrig{*}(U_s^{lisse}/\breve{K}_N)\cong \hrig{*}(\dl^d_{\bar{\F}}/\breve{K}_N)\]  Comme l'identification entre les fibres spéciales est  $\gln_{d+1}(\OC_K)$-équivariant, le morphisme au niveau des cohomologies l'est aussi.

L'isomorphisme de l'énoncé s'obtient en composant chacune de ces bijections équivariantes intermédiaires. 
\end{proof}


\section{Actions de groupes sur la partie lisse\label{sssectionlt1dlacgr}} 

Nous notons $N= q^{d+1}-1$ et $K_N=K(\varpi_N)$ o\`u $\varpi_N$ est une racine $N$-i\`eme de $\varpi$ et $\Omega^d_{\bar{\F}}=\P^d \setminus \bigcup_{H \in \HC_1} H=\spec(A)$ et ${\rm DL}^d_{\bar{\F}}=\spec(B)$ . 

L'interprétation modulaire de $Z_0\otimes \OC_C$ fournit une action naturelle des trois groupes $\OC_D^*$, $G^{\circ}$  et de $I_K$ sur cet espace. Par naturalité du procédé d'éclatement, ces actions se prolongent à chaque modèle $Z_i$ et les fl\`eches $p_i : Z_i \to Z_0$ sont \'equivariantes. Ces actions ont pour effet de permuter les composantes irréductibles de la fibre spéciale et leurs intersections à savoir les  fermés de la forme $Y_{M,i}$ et ces transformations  respectent la dimension des espaces $M$. Ainsi, les trois groupes $G^{\circ}$, $\OC_D^*$ et $I_K$ laissent stable la composante $Y_{\{0\} ,i}$ et permutent les autres, l'action de se transporte à la composante ouverte (\ref{TheoZH} 3.) $Y_{\{0\} ,d}^{lisse}=\cdots=Y_{\{0\} ,0}^{lisse}$ et donc à la variété de Deligne-Lusztig ${\rm DL}^d_{\bar{\F}}$ d'après le théorème  \ref{TheoZH} 4. Mais, si l'on fixe des identifications $I_K/ I_{K_N} \cong\F_{q^{d+1}}^* \cong  \OC_D^*/1+ \Pi_D \OC_D$ (on rappelle que $\F_{q^{d+1}}^* $ est isomorphe à $\gal(\dl_{\bar{\F}}^d/ \Omega_{\bar{\F}}^d)$), on obtient alors une autre action des trois groupes sur ${\rm DL}^d_{\bar{\F}}$. 

\begin{theo}\label{theodlactgdw}
Sous le choix d'identification convenable $I_K/ I_{K_N} \cong\F_{q^{d+1}}^* \cong  \OC_D^*/1+ \Pi_D \OC_D$, les différentes actions décrites plus haut de $G^{\circ}$, $\OC_D^*$ et $I_K$ sur ${\rm DL}^d_{\bar{\F}}$ co\"incident.  
\end{theo}

\begin{proof}
Pour l'action de $G^{\circ}$, cela d\'ecoule clairement de l'action de $\gln_{d+1}(\F)$ sur les composantes irr\'eductibles $Y_a$ \ref{TheoZH}4.. Le mod\`ele lisse dont la fibre sp\'eciale est isomorphe \`a ${\rm DL}^d_{\bar{\F}}$ est obtenu en \'etendant les scalaires \`a $\OC_{\breve{K}_N}$ puis en normalisant dans $\lt^1 \otimes \breve{K}_N$. Cette op\'eration a pour effet de changer les variables  $X_0, \dots, X_d$ d\'efinies dans la partie \ref{sssectionltneq} en les variables  $\frac{X_0}{\varpi_N}, \dots, \frac{X_d}{\varpi_N}$. Ceci explique le fait que l'action de $I_K$ est triviale sur $I_{K_N}$ et que cette action identifie $I_K/ I_{K_N} \cong \gal({\rm DL}^d_{\bar{\F}}/ \Omega^d_{\bar{\F}}) \cong \F_{q^{d+1}}^*$ (cf \cite[Proposition 5.5. (ii)]{yosh}). 
  
Le point le plus d\'elicat est la description de l'action de $\OC_D^*$. 
Celle-ci commute \`a l'action de $I_K$ qui  agit par automorphismes du groupe de Galois $\gal(\dl_{\bar{\F}}^d/ \Omega_{\bar{\F}}^d)$. Comme $A= \Of(\Omega^d_{\bar{\F}})$ est pr\'ecis\'ement l'ensemble des fonctions invariantes sous ce groupe, $\OC_D^*$ pr\'eserve $A$. La restriction de cette action \`a $A$ d\'efinit des automorphismes $\gln_{d+1}(\F)$-\'equivariants, ils sont triviaux sur $A$ par \ref{lemautdl}. Ainsi, l'action étudiée d\'efinit un morphisme $\OC_D^* \to \gal(\dl_{\bar{\F}}^d/ \Omega_{\bar{\F}}^d)$ qui est trivial sur $1+ \Pi_D \OC_D$ car c'est un pro-$p$-groupe 
qui s'envoie sur un groupe cyclique d'ordre premier \`a $p$ d'où une flèche $  \OC_D^*/(1+ \Pi_D \OC_D) \to \gal(\dl_{\bar{\F}}^d/ \Omega_{\bar{\F}}^d)$. Il s'agit de voir que c'est un isomorphisme voire une injection, par \'egalit\'e des cardinaux.

D\'eployons $D^*$ dans $\gln_{d+1}(K_{(d+1)})$ avec $K_{(d+1)}$ l'extension non-ramifiée de degré $d+1$. Prenons $b \in \OC_D^*/(1+ \Pi_D \OC_D)$ et relevons-le en $\tilde{b} \in \OC_D^*$ r\'egulier elliptique. En effet, les \'el\'ements r\'eguliers elliptiques forment un ouvert Zariski de $\OC_D^*$ et sont donc denses pour la topologie $p$-adique. Appelons $\iota(b)$ l'image de $b$ dans $\gal(\dl_{\bar{\F}}^d/ \Omega_{\bar{\F}}^d)$. La description explicite de la cohomologie des vari\'et\'es de Deligne-Lusztig nous donne\footnote{On a $\tr(\iota(b) | \hetc{\heartsuit}(\dl_{\bar{\F}}^d,\bar{\Q}_l)[\theta])=\theta(\iota(b))\dim\hetc{\heartsuit}(\dl_{\bar{\F}}^d,\bar{\Q}_l)[\theta]=\theta(\iota(b))\dim\hetc{\heartsuit}(\Omega_{\bar{\F}}^d,\bar{\Q}_l)$ pour tout caractère $\theta$. Ainsi, $\tr(\iota(b) | \hetc{\heartsuit}(\dl_{\bar{\F}}^d,\bar{\Q}_l))=\dim\hetc{\heartsuit}(\Omega_{\bar{\F}}^d,\bar{\Q}_l)\sum_{\theta}\theta(\iota(b))$ et la dernière somme est nulle ssi $\iota(b)\neq 1$}
\[ \iota(b)=1 \Leftrightarrow \tr(\iota(b) | \hetc{\heartsuit}(\dl_{\bar{\F}}^d,\bar{\Q}_l)) \neq 0\]
Mais on a la formule de trace \cite[Th\'eor\`eme (3.3.1)]{stra} 
\[ \tr(\iota(b) | \hetc{\heartsuit}(\dl_{\bar{\F}}^d, \bar{\Q}_l)) = \tr(\tilde{b} | \hetc{\heartsuit}(\lt^{1}, \bar{\Q}_l))= | \fix(\tilde{b}, \lt^1(C)) |  \]
La premi\`ere \'egalit\'e d\'ecoule de l'isomorphisme entre les cohomologies de \cite{yosh}. 

Nous expliquons comment calculer le nombre de points fixes sur $\lt^1(C)$ d'un élément  $\tilde{b}$  r\'egulier elliptique. Nous avons un morphisme des p\'eriodes analytique $\OC_D^*$-\'equivariant $\pi_{GH}: \lt^1\to \lt^0 \to \P(W)$ \cite{grho} o\`u $W$ s'identifie \`a $C^{d+1}$ et l'action sur $W$ se d\'eduit du choix d'un isomorphisme $D \otimes C \cong \mat_{d+1}(C)$. Ainsi $\tilde{b}$ admet $d+1$ points fixes dans $\P(W)$ qui sont les droites propres. Fixons $x \in \P(W)$ une de ces droites propres. D'apr\`es \cite[Proposition (2.6.7)(ii)-(iii)]{stra}, il existe alors  $g_{\tilde{b}} \in G=\gln_{d+1}(K)$ ayant m\^eme polyn\^ome caract\'eristique que $\tilde{b}$ tel que 
\[ \pi_{GH}^{-1}(x) \cong G/G_1\varpi^{\Z} \text{ et } b.(hG_1\varpi^{\Z}) = (g_bh)G_1\varpi^{\Z} \]
avec $G_1=1+ \varpi \mat_{d+1}(\OC_K)$. Ainsi $hG_1\varpi^{\Z}$ est un point fixe si et seulement si $h^{-1}g_{\tilde{b}}h \in G_1\varpi^{\Z}$.

Supposons maintenant $b\in \OC_D^*/(1+ \Pi_D \OC_D)$ non-trivial et montrons $| \fix(\tilde{b}, \lt^1(C)) |=0$ pour $\tilde{b}$ un relèvement r\'egulier elliptique. 
Soit  $x \in \P(W)$ une droite propre pour $\tilde{b}$, la matrice $g_{\tilde{b}}\in G$ décrivant l'action de $\tilde{b}$ sur $\pi_{GH}^{-1}(x)$ et supposons l'existence d'un point fixe  $hG_1\varpi^{\Z}\in\pi_{GH}^{-1}(x)$ pour $\tilde{b}$ . Dans ce cas, $h^{-1}g_{\tilde{b}}h\in G_1\varpi^{\Z}$ est diagonalisable. Comme $b \equiv \zeta \pmod{\Pi_D}$, o\`u $\zeta$ est une racine de l'unit\'e diff\'erente de $1$ dans $K_{(d+1)}$, au moins une valeur propre est dans $\zeta+ \mG_{C}$\footnote{On voit que que le produit de toutes les valeurs propres de $\tilde{b}-\zeta\in \Pi_D\OC_D$ est $\nr(\tilde{b}-\zeta)\in \mG_C$. Ainsi, une de ces valeurs propres doit être dans $\mG_C$. }. De m\^eme, une matrice dans $G_1\varpi^{\Z}$ diagonalisable a ses valeurs propres dans $\varpi^{\Z}(1+ \mG_{C})$\footnote{Il s'agit de voir que les matrices diagonalisables $M$ dans $\mat_{d+1}(\OC_K)$ ont des valeurs propres dans $\OC_C$. Raisonnons sur  une extension finie $L$ dans laquelle on peut diagonaliser $M$. Prenons un vecteur propre $v$ de $M$ que l'on suppose unimodulaire quitte à le normaliser ie $v$ a au moins une coordonnée dans $\OC_L^*$. Par Nakayama topologique, on peut trouver une $\OC_L$-base du réseau standard $\OC_L^{d+1}\subset L^{d+1}$ contenant $v$. Comme la matrice $M$ préserve ce réseau, on a $Mv=\lambda v\in\OC_L^{d+1}$ d'où $\lambda\in\OC_L$.  }. L'\'el\'ement $h^{-1}g_bh$ ne peut v\'erifier ces deux conditions en m\^eme temps, ce qui montre qu'il ne peut y avoir de point fixe dans $\pi_{GH}^{-1}(x)$ ni m\^eme dans $\lt^1(C)$.   Ainsi, $\tr(\iota(b) | \hetc{\heartsuit}(\dl_{\bar{\F}}^d, \bar{\Q}_l))=0$ et $\iota$ est injective donc bijective. 
\end{proof}

\begin{rem}
On a aussi une donnée de descente sur $\widehat{\lt}^1$ à la Weil (construite dans \cite[(3.48)]{rapzin}) qui induit sur $\dl_{\bar{\F}}^d$ la donnée de descente provenant de la forme $\F$-rationnelle  $\dl_{\F}^d$ en suivant l'argument \cite[Lemme (3.1.11)]{wa}. 
\end{rem}

D'après la remarque précédente, la structure $\F_q$-rationnelle $\dl_{\F_q}^d$ de $\dl_{\bar{\F}_q}^d$ induit une action du Frobenius $\varphi$ sur $\hetc{*}(\dl_{\bar{\F}_q}^d, \bar{\Q}_l)$ et $\hrigc{*}(\dl_{\F_q}^d/K_0) \otimes_{K_0} \bar{K}$ (cf \cite[2.1 Proposition]{EtSt} 
 pour ce dernier). De plus,  $\varphi^{d+1}$  commute aux actions de $\gln_{d+1}(\F_q) \times \F^*_{q^{d+1}}$ et les $\gln_{d+1}(\F_q)$-repr\'esentations $\hetc{d}(\dl_{\F_q}^d, \bar{\Q}_l)[\theta]$ et $\hrigc{*}(\dl_{\F_q}^d)[ \theta]$ sont irr\'eductibles pour $\theta$ primitif. Cet op\'erateur agit alors par un scalaire $\lambda_{\theta}$. On a le r\'esultat suivant : 

\begin{prop}\label{propdlwk}
$\lambda_{\theta}= (-1)^d q^{\frac{d(d+1)}{2}}$ pour la cohomologie \'etale $l$-adique et pour la cohomologie rigide. 
\end{prop}

\begin{proof}
L'argument provient de \cite[section V, 3.14]{DM} \`a quelques twists pr\`es. Nous avons pr\'ef\'er\'e r\'ep\'eter la preuve pour \'eviter toute confusion. Pour simplifier, dans toute la d\'emonstration, $\hhh$ d\'esignera la cohomologie consid\'er\'ee ($l$-adique ou rigide), $\hhh^{\heartsuit}$ la caract\'eristique d'Euler vue comme une repr\'esentation virtuelle, $Y=\dl_{\bar{\F}}^d$ et $X=\Omega_{\bar{\F}}^d $. Par d\'efinition, on a \[ \lambda_{\theta}= \frac{{\rm Tr}(\varphi^{d+1} | \hhh^{d}(Y)[\theta])}{\dim \hhh^d(Y)[\theta]}= (-1)^d \frac{{\rm Tr}(\varphi^{d+1} | \hhh^{\heartsuit}(Y)[\theta])}{\dim \hhh^\heartsuit(Y)[\theta]}. \] 
On dispose du projecteur $p_{\theta}= \frac{1}{N} \sum_{a \in \mu_N} \theta(a^{-1})a$ sur la partie $\theta$-isotypique d'o\`u la suite d'\'egalit\'es 
\begin{equation}
\label{eqtrace}
{\rm Tr}(\varphi^{d+1} | \hhh^{\heartsuit}(Y)[\theta])={\rm Tr}(\varphi^{d+1} p_{\theta} | \hhh^{\heartsuit}(Y))= \frac{1}{N} \sum_{a \in \mu_N} \theta(a^{-1}) {\rm Tr}(a\varphi^{d+1} | \hhh^{\heartsuit}(Y)). 
\end{equation}
Mais d'apr\`es le th\'eor\`eme des points fixes de Lefshetz (cf \cite[6.2 Théorème]{EtSt} pour la cohomologie rigide), ${\rm Tr}(a\varphi^{d+1} | \hhh^{\heartsuit}(Y))= | Y^{a \varphi^{d+1}} |$. Nous chercherons \`a d\'eterminer ces diff\'erents cardinaux. Remarquons l'\'egalit\'e ensembliste suivante due \`a la $\varphi$-\'equivariance de la projection $\pi : Y \to X$ : 
\[ \pi^{-1}(X^{\varphi^{d+1}})= \pi^{-1}(X(\F_{q^{d+1}}))=\bigcup_{a} Y^{a \varphi^{d+1}}    \] où la dernière union est disjointe.
Pour le sens indirect, il suffit d'observer que pour un point ferm\'e $y \in Y^{a \varphi^{d+1}}$, $\pi(y)=\pi(a \varphi^{d+1}(y))=\varphi^{d+1}(\pi(y))$. Pour le sens direct, on raisonne de la m\^eme mani\`ere en se donnant un point ferm\'e $y$ tel que $\pi(y)=\varphi^{d+1}(\pi(y))$. Comme le groupe de Galois du rev\^etement $\pi$ agit librement et transitivement sur les fibres, il existe un unique $a$ dans $\mu_N$ tel que $a \varphi^{d+1}(y)=y$. 

Nous allons prouver que $\pi^{-1}(X(\F_{q^{d+1}}))= Y(\F_{q^{d+1}})(= |Y^{\varphi^{d+1}}|)$, ce qui montrera l'annulation des autres ensembles de points fixes. Nous venons de montrer l'inclusion indirecte et \'etablissons l'autre inclusion. Prenons $(x_1, x_2, \dots, x_d, t)$ un point de $ \pi^{-1}(X(\F_{q^{d+1}}))$. Ainsi $(x_1, \dots, x_d)$ appartient \`a $\F_{q^{d+1}}^d$ et $t$ v\'erifie l'\'equation\footnote{On voit $Y$ comme le rev\^etement de type Kummer associ\'e \`a $u$.} $t^N-u(x_1, \dots, x_d)=0$  dans une extension finie de $\F_{q^{d+1}}$. Mais le polyn\^ome $T^N-u(x_1, \dots, x_d) \in \F_{q^{d+1}}[T]$ est scind\'e \`a racines simples dans $\F_{q^{d+1}}$ et $(x_1, \dots, x_d,t)$ est $\F_{q^{d+1}}$-rationnel. On a montr\'e l'\'egalit\'e voulue et on obtient de plus
\[ | \pi^{-1}(X(\F_{q^{d+1}})|=|Y(\F_{q^{d+1}})|= N | X(\F_{q^{d+1}}) |. \] 
Il reste \`a calculer cette derni\`ere quantit\'e. En fixant une $\F_q$ base de $\F_{q^{d+1}}$, on a un isomorphisme de $\F_q$-espaces vectoriels $\F_{q^{d+1}}^{d+1} \xrightarrow{\sim} M_{d+1}(\F_q)$ (on voit chaque \'el\'ement de $\F_{q^{d+1}}$ comme un vecteur colonne). Il est ais\'e de voir qu'un vecteur $x \in \F_{q^{d+1}}^{d+1} \backslash \{ 0 \}$ engendre une droite de $X(\F_{q^{d+1}})$ si et seulement si la matrice associ\'ee par l'isomorphisme pr\'ec\'edent est inversible, d'o\`u $|X(\F_{q^{d+1}})|= \frac{| \gln_{d+1}(\F_q)|}{N}$.     
En rempla\c{c}ant dans \eqref{eqtrace}, on obtient : 
\[ {\rm Tr}(\varphi^{d+1} | \hhh^{\heartsuit}(Y)[\theta]) =\frac{ | \gln_{d+1}(\F_q) |}{N} = (q^{d+1}-q^d) \dots (q^{d+1}-q) \]
d'o\`u $\lambda_{\theta}=(-1)^d q^{\frac{d(d+1)}{2}}$. 
\end{proof}

\section{Réalisation de la correspondance de Langlands locale}




Dans cette partie, nous allons d\'ecrire la cohomologie des espaces $\MC_{LT}^1$ et montrer qu'elle r\'ealise la correspondance de Jacquet-Langlands. On \'etendra les scalaires \`a $C$ pour tous les espaces consid\'er\'es en fibre g\'en\'erique. 


 On pourra simplifier le produit $G^{\circ} \times D^* $ (resp. $G^{\circ} \times D^* \times W_K$) en $GD$ (resp. $GDW$). On a  
une "valuation" $v_{GDW}$ sur $GDW$  :
\[ v_{GDW} : (g,b,w) \in GDW \mapsto v_K( \nr(b) \art^{-1}(w)) \in \Z. \]
On introduit alors pour  $i=0$ ou $i=d+1$, $[GDW]_{i}=v_{GDW}^{-1}(i\Z)$ et  $[D]_{i}=D^*\cap[GDW]_{i}=\OC_D^*\varpi^\Z$, $[W]_{i}=W_K\cap[GDW]_{i}=I_K(\varphi^{d+1})^\Z$.
 

  L'espace  $\MC_{LT}^1$ (resp. $\MC_{LT}^1/ \varpi^{\Z}$) s'identifie non canoniquement \`a $\lt^1 \times \Z$ (resp. $\lt^1 \times \Z/(d+1)\Z$) et on confondra alors $\lt^1$ avec $\lt^1 \times \{0 \}$ (où $0$ est vu dans $\Z$ ou dans $\Z/(d+1)\Z$ suivant si $\lt^1$ est vu comme un sous-espace de $\MC_{LT}^1$ ou de $\MC_{LT}^1/ \varpi^{\Z}$. Chaque élément $(g,b,w)\in GDW$ (resp $GD$) envoie $\lt^1 \times \{i \}$ si $\lt^1 \times \{i+ v_{GDW} (g,b,w)\}$
On obtient alors  une action sur $\lt^1$ de $[GDW]_{d+1}$.  On a alors  la relation : 
\[ \hetc{i}( \MC_{LT}^1/ \varpi^{\Z},\bar{\Q}_l) \underset{GDW}{\cong} \cind^{GDW}_{[GDW]_{d+1}}  \hetc{i}( \lt^1,\bar{\Q}_l)\cong \hdrc{i}( \lt^1,\bar{\Q}_l)^{d+1}. \]
\[ \hdrc{i}( \MC_{LT}^1/ \varpi^{\Z}) \cong \hdrc{i}( \lt^1)^{d+1}. \]

Passons aux repr\'esentations qui vont nous int\'eresser. Nous d\'efinissons d'abord des $C$-repr\'esentations sur $[GDW]_{d+1}$ que nous \'etendrons \`a $GDW$ par induction. Fixons $\theta$ un caract\`ere primitif de $\F_{q^{d+1}}^*$ et des isomorphismes $\OC_D^*/1+ \Pi_D \OC_D \cong \F_{q^{d+1}}^* \cong I_K/I_{K_N}$ o\`u $K_N= K(\varpi^{\frac{1}{N}})$ avec $N=q^{d+1}-1$. On pose : 
\begin{itemize}
\item $\theta$ sera vu comme une $[D]_{d+1}$-repr\'esentation via $[D]_{d+1}=\OC_D^* \varpi^{\Z} \to \OC_D^* \to \F_{q^{d+1}}^*$, 
\item $\bar{\pi}_{\theta}$ sera la repr\'esentation associ\'ee \`a $\theta$ sur $\bar{G}=\gln_{d+1}(\F)$ via la correspondance de Green. On la voit comme une $G^{\circ} \varpi^{\Z}$-repr\'esentation via $G^{\circ} \varpi^{\Z} \to G^{\circ} \to \bar{G}$,
\item $\tilde{\theta}$ sera la repr\'esentation de $[W]_{d+1}$ telle que $\tilde{\theta}|_{I_K}=\theta$ via $I_K \to I_K/I_{K_N} \xrightarrow{\sim} \F_{q^{d+1}}^*$ et $\tilde{\theta}(\varphi^{d+1})=(-1)^d q^{\frac{d(d+1)}{2}}$. 
\end{itemize}  
Par induction, on obtient : 
\begin{itemize}

\item une $D^*$-repr\'esentation $\rho(\theta):= \cind_{[D]_{d+1}}^{D^*} \theta$, 
\item une $W_K$-repr\'esentation $\sigma^{\sharp}(\theta):= \cind_{[W]_{d+1}}^{W_K} \tilde{\theta}$.   
\end{itemize}

Nous souhaitons prouver :

\begin{theo}
Fixons un isomorphisme $C\cong \bar{\Q}_l$. Si $\theta$ un caract\`ere primitif, on a : 
\[ \homm_{D^*}(\rho(\theta),\hdrc{i}( (\MC_{LT}^1/ \varpi^{\Z})/C))\underset{G^{\circ}}{\cong}  \begin{cases} \bar{\pi}_{\theta}^{d+1} & \text{ si } i=d \\ 0 & \text{ sinon.} \end{cases} \]
\[ \homm_{D^*}(\rho(\theta),\hetc{i}( (\MC_{LT}^1/ \varpi^{\Z}),\bar{\Q}_l)\otimes C)\underset{G^{\circ} \times W_K}{\cong}  \begin{cases} \bar{\pi}_{\theta} \otimes \sigma^{\sharp}(\theta) & \text{ si } i=d \\ 0 & \text{ sinon.} \end{cases} \]     
\end{theo}

\begin{rem}

Pour la deuxième partie, il a été prouvé dans \cite[Proposition 6.14.]{yosh} le résultat moins précis  \[\homm_{I_K}(\tilde{\theta},\hetc{i}( (\MC_{LT}^1/ \varpi^{\Z}),\bar{\Q}_l)\otimes C)\underset{G^{\circ} }{\cong}  \begin{cases} \bar{\pi}_{\theta}^{d+1} & \text{ si } i=d \\ 0 & \text{ sinon.} \end{cases}.\] L'énoncé que nous obtenons est une conséquence directe de cette égalité et de \ref{theolt1dr} et \ref{theodlactgdw}. 

\end{rem}

\begin{proof}
Nous avons d\'ej\`a prouv\'e l'annulation de la cohomologie quand $i \neq d$. Int\'eressons-nous au cas $i=d$. D'après les discussion précédente, on a $\hdrc{i}( \MC_{LT}^1/ \varpi^{\Z}) \cong \hdrc{i}( \lt^1)^{d+1}$ et l'action de $[D]_{d+1}\times G^{\circ}$ respecte cette décomposition en produit. On a alors des isomorphismes $G^{\circ}$-équivariant:
\begin{align*}
\homm_{D^*}(\rho(\theta), \hdrc{d}( (\MC_{LT}^1/ \varpi^{\Z})/C)) & = \homm_{D^*}(\cind_{[D]_{d+1}}^{D^*} \theta, \hdrc{d}( (\MC_{LT}^1/ \varpi^{\Z})/C)) \\
& = \homm_{[D]_{d+1}}(\theta,  \hdrc{d}( (\MC_{LT}^1/ \varpi^{\Z})/C)) \\
& =  \homm_{\F_{q^{d+1}}^*}(\theta,  \hrigc{d}( \dl_{\bar{\F}}/C))^{d+1} &\text{d'apr\`es \ref{theolt1dr} et \ref{theodlactgdw}} \\
& =  \bar{\pi}_{\theta}^{d+1}. 
\end{align*}

Le même raisonnement en cohomologie étale entraîne d'après \cite[Proposition 6.16.]{yosh} et \ref{theodlactgdw} \[\homm_{D^*}(\rho(\theta), \hetc{d}( (\MC_{LT}^1/ \varpi^{\Z}),\bar{\Q}_l)\otimes C) =  \bar{\pi}_{\theta}^{d+1}\] en tant que $G^{\circ}$-repr\'esentation. Plus précisément, on a 

\begin{align*}
\homm_{D^*}(\rho(\theta), \hetc{d}( (\MC_{LT}^1/ \varpi^{\Z}),\bar{\Q}_l)\otimes C) & = \homm_{D^*}(\cind_{[D]_{d+1}}^{D^*} \theta,  \cind^{GDW}_{[GDW]_{d+1}}  \hetc{d}( \lt^1,\bar{\Q}_l)\otimes C) \\
& = \homm_{[D]_{d+1}}(\theta,  \cind^{GW}_{[GW]_{d+1}}  \hetc{d}( \lt^1,\bar{\Q}_l)\otimes C) \\
& = \cind^{GW}_{[GW]_{d+1}} \homm_{\F_{q^{d+1}}^*}(\theta,  \hetc{d}(\dl_{\bar{\F}},\bar{\Q}_l)\otimes C ) &\\
& =  \cind^{GW}_{[GW]_{d+1}} \hetc{d}(\dl_{\bar{\F}},\bar{\Q}_l)[\theta]\otimes C=:\tau_{GW}. 
\end{align*}
en tant que $GW$-repr\'esentation.

De plus, on a 
\begin{align*}
\homm_{G^{\circ}}(\bar{\pi}_{\theta}, \tau_{GW}) & = \homm_{G^{\circ}}(\bar{\pi}_{\theta},\cind^{W_K}_{[W]_{d+1}}  \hetc{d}(\dl_{\bar{\F}},\bar{\Q}_l)[\theta]\otimes C) \\
&= \cind^{W_K}_{[W]_{d+1}} \homm_{G^{\circ}}(\bar{\pi}_{\theta}, \hetc{d}(\dl_{\bar{\F}},\bar{\Q}_l)[\theta]\otimes C) \\
&= \cind^{W_K}_{[W]_{d+1}} \tilde{\theta} = \sigma^{\sharp}(\theta) &\text{ d'apr\`es \ref{theodlactgdw},}
\end{align*}
en tant que $W_K$-repr\'esentation, ce qui conclut la preuve. 
\end{proof}

\nocite{J1},\nocite{J2}
\bibliographystyle{alpha}
\bibliography{drlt_v1}

\end{document}